\documentclass[11pt, reqno]{amsart}
\usepackage{amsthm, mathrsfs, amsmath, amstext, esint, amsxtra, amsfonts, slashed, dsfont, amssymb}
\usepackage{cite, relsize}
\usepackage[scale=0.75, centering, headheight=14pt]{geometry}
\usepackage[latin1]{inputenc}
\usepackage{microtype}
\usepackage{color,enumitem,graphicx}
\usepackage[colorlinks=true,urlcolor=blue, citecolor=red,linkcolor=blue,
linktocpage,pdfpagelabels, bookmarksnumbered,bookmarksopen]{hyperref}
\usepackage[hyperpageref]{backref}
\usepackage[english]{babel}
\usepackage{amsrefs}

\numberwithin{equation}{section}

\newtheorem{theorem}{Theorem}[section]
\newtheorem{lemma}[theorem]{Lemma}
\newtheorem{proposition}[theorem]{Proposition}
\newtheorem{corollary}[theorem]{Corollary}

\newtheorem*{theorem*}{Conjecture}

\newtheoremstyle{remarkstyle}
{}{}{}{ }{\bfseries}{.}{ }{\thmname{#1}\thmnumber{ #2}\thmnote{ (#3)}}
\theoremstyle{remarkstyle}
\newtheorem{remark}{Remark}[section]
\newtheorem{definition}{Definition}[section]

\def\({\left(}
\def\){\right)}
\def\<{\left\langle}
\def\>{\right\rangle}

\newcommand{\De}{\slashed{\nabla}}

\newcommand{\Z}{\mathbb Z}

\newcommand{\R}{\mathbb R}
\newcommand{\C}{\mathbb C}
\newcommand{\Ac}{\mathcal A}

\newcommand{\Sc}{\mathcal S}
\newcommand{\Gc}{\mathcal G}

\newcommand{\Oc}{\mathcal O}

\newcommand{\Mcal}{\mathcal M}
\newcommand{\vareps}{\varepsilon}
\DeclareMathOperator*{\loc}{loc}

\DeclareMathOperator*{\opt}{opt}

\DeclareMathOperator*{\sd}{sd}
\DeclareMathOperator*{\rea}{Re}
\DeclareMathOperator*{\ima}{Im}

\DeclareMathOperator{\RE}{Re}
\DeclareMathOperator{\IM}{Im}

\title[NLS system in $\chi^3$ media]
{Sharp conditions for scattering and blow-up for a system of NLS arising in optical materials with $\chi^3$ nonlinear response}

\author[A. H. Ardila]{Alex H. Ardila}
\address[A. H. Ardila]{Universidade Federal de Minas Gerais\\ ICEx-UFMG\\ CEP 30123-970\\ MG, Brazil} 
\email{ardila@impa.br}

\author[V. D. Dinh]{Van Duong Dinh}
\address[V. D. Dinh]{Laboratoire Paul Painlev\'e UMR 8524, Universit\'e de Lille CNRS, 59655 Villeneuve d'Asc, France
and 
Department of Mathematics, HCMC University of Pedagogy, 280 An Duong Vuong, Ho Chi Minh, Vietnam}
\email{contact@duongdinh.com}

\author[L. Forcella]{Luigi Forcella}
\address[L. Forcella]{Department of Mathematics, Heriot-Watt University and The Maxwell Institute for the Mathematical Sciences, Edinburgh, EH14 4AS, United Kingdom}
\email{l.forcella@hw.ac.uk}

\subjclass[2010]{35B44; 35Q55}
\keywords{Nonlinear Schr\"odinger systems, Cubic-type interactions, Scattering, Blow-up, Morawetz estimates}

\begin{document}
	
\begin{abstract}
We study the asymptotic dynamics  for solutions to a system of nonlinear  Schr\"odinger equations with cubic interactions, arising in nonlinear optics. We provide sharp threshold criteria leading to global well-posedness and scattering of solutions, as well as formation of singularities in finite time for (anisotropic) symmetric initial data. The free asymptotic results are proved by means of Morawetz and interaction Morawetz estimates. The blow-up results are shown by combining variational analysis and an ODE argument, which overcomes the unavailability of the convexity argument based on virial-type identities. 

\end{abstract}

\maketitle
	
\section{Introduction}
\label{Intro}
\setcounter{equation}{0}

In this paper, we consider the Cauchy problem for the following system of nonlinear Schr\"odinger equations with cubic interaction 
\begin{equation}\label{SNLS}
\left\{
\begin{aligned}
i\partial_t u + \Delta u - u &= -\left(\dfrac{1}{9} |u|^2 + 2|v|^2 \right) u - \dfrac{1}{3} \overline{u}^2 v, \\
i \gamma \partial_t v +  \Delta v - \mu v &= -\left( 9 |v|^2 + 2 |u|^2\right) v - \dfrac{1}{9} u^3, 
\end{aligned}
\right.
\end{equation}
with initial datum
$\left.(u, v)\right|_{t=0} =(u_0,v_0).$ Here $u, v: \R \times \R^3 \rightarrow \C$, $u_0,v_0: \R^3 \rightarrow \C$, and the parameters $\gamma, \mu$ are strictly positive real numbers.

The system \eqref{SNLS} is the dimensionless form of a system of nonlinear Schr\"odinger equations as derived in \cite{SBK-OL} (see also \cite{SBK-JOSA}), where the interaction between an optical beam at some fundamental frequency and its third harmonic is investigated. More precisely, from a physical point of view, \eqref{SNLS} models the interplay of an optical monochromatic beam with its third harmonic in a Kerr-type medium (we refer to \cite{OP} for the latter terminology, as well as for a sketch of the derivation of \eqref{SNLS}).

Models such as in \eqref{SNLS} arise in nonlinear optics in the context of the so-called cascading nonlinear processes. These processes can indeed generate effective higher-order nonlinearities, and they stimulated the study of spatial solitary waves in  optical materials with $\chi^2$ or $\chi^3$ susceptibilities (or nonlinear response, equivalently). 

Let us mention, following \cite{CdMS}, the difference between $\chi^2$ (quadratic) and $\chi^3$ (cubic) media. The contrast basically reflects the order of expansion (in terms of the electric field) of the polarization vector, when decomposing the electrical induction field appearing in the Maxwell equations as the sum of the electric field $\mathbb E$ and the polarization vector $\mathbb P$. Indeed, for ``small'' intensities of the electric field, the polarization response is linear, while for ``large'' intensities of $\mathbb E,$  the vector $\mathbb P$ has a non-negligible nonlinear component, denoted by $\mathbb P_{nl}$. Thus, when considering the Taylor expansion for $\mathbb P_{nl}$, one gets the presence of (at least) quadratic and cubic terms whose coefficients $\chi^j$, which depend on the frequency of the electric field $\mathbb E,$ are called $j$-th optical susceptibility. For $j=2,3,$ they are usually denoted by $\chi^2$ and $\chi^3$.  Therefore quadratic media arise from approximation of the type $\mathbb P_{nl}\sim \chi^2\mathbb E^2,$ and similarly one can define cubic media. The so-called non-centrosymmetric crystals are typical examples of $\chi^2$ materials. Moreover, it can be shown, see \cite{Fib}, that isotropic materials have $\chi^{2n}=0$ susceptibility, namely even orders of nonlinear responses are zero. In the latter case, the leading-order in the expansion of $\mathbb P_{nl}$ is cubic, and these kind of isotropic materials are called Kerr-materials. See the monographs \cite{Fib,SS, Boy} for more discussions. In addition, we refer to \cite{AP, BDST, B99, CdMS, Kivshar, LGT, SBK-OL, SBK-JOSA,ZZS}, and references therein, for more insights on physical motivations and physical results (both theoretical and numerical) about \eqref{SNLS} and other NLS systems with cubic and quadratic interactions. Models as in \eqref{SNLS} are therefore physically relevant, and they deserve a rigorous mathematical investigation. In particular, we are interested in qualitative properties of solutions to \eqref{SNLS}.\\

Our main goal is to understand the asymptotic dynamics of solutions to \eqref{SNLS}, by establishing conditions ensuring global existence and their long time behavior, or leading to formation of singularities in finite time. \\
Let us mention since now on, that once the Strichartz machinery has  been established, and this is nowadays classical, local well-posedness of \eqref{SNLS} at the energy regularity level (i.e. $H^1(\R^3),$ mathematically speaking) is relatively straightforward to get (see below for a precise definition of the functional space to employ a fixed point argument).

The dynamics of solution of NLS-type equation is intimately related to the existence of ground states (see below for a more precise definition). The analysis of solitons is a very important physical problem, and the main difference between $\chi^2$ media and $\chi^3$ media, is that, in the latter case, the cubic nonlinearity is $L^2$ supercritical, while in the former quadratic nonlinearities are $L^2$ subcritical. The last two regimes dramatically reflect the possibility for the problem to be globally well-posed, and the stability/instability properties of the solitons are different. See \cite{CdMS} for further discussions, and a rigorous analysis for solitons in quadratic media. \\

Regarding system \eqref{SNLS}, existence of ground states and their instability properties were established in a recent paper by Oliveira and Pastor, see \cite{OP}. Our aim is to push forward their achievements to obtain a qualitative description of solutions to \eqref{SNLS}, by giving sharp thresholds, defined by means of quantities linked to the ground state, are sufficient to  guarantee a linear asymptotic dynamics for large time (i.e. scattering) or  finite time blow-up of the solutions.\\

Let us start our rigorous mathematical discussion about \eqref{SNLS}.
The existence of solutions is quite simple to obtain. As said above, it is well-known that \eqref{SNLS} is locally well-posed in $H^1(\R^3) \times H^1(\R^3),$ (see e.g., \cite{Cazenave}). More precisely, for $(u_0,v_0) \in H^1(\R^3)\times H^1(\R^3)$, there exist $T_{\pm}>0$ and a unique solution $(u,v)\in X((-T_{-},T_{+})) \times X((-T_{-},T_{+}))$, where
\[
X((-T_{-},T_{+})):= C((-T_{-},T_{+}), H^1(\R^3)) \cap L^q_{\loc}((-T_{-}, T_{+}), W^{1,r}(\R^3))
\]
for any Strichartz $L^2$-admissible pair $(q,r)$, i.e., $\frac{2}{q}+\frac{3}{r}=\frac{3}{2},$ for  $2\leq r \leq 6.$ See Section \ref{sec:pre}.
In addition, the maximal times of existence obey the blow-up alternative, i.e., either $T_{+}=\infty$, or $T_{+}<\infty$ and $\lim_{t\nearrow T_{+}} \|(u(t),v(t))\|_{H^1(\R^3)\times H^1(\R^3)} =\infty$, and similarly for $T_{-}$. When $T_{\pm} =\infty$, we call the solution global.
Solutions to \eqref{SNLS} satisfy conservation laws of mass and energy, namely
\begin{align*}
M_{3\gamma}(u(t),v(t)) &= M_{3\gamma}(u_0,v_0), \tag{Mass} \\
E_\mu(u(t),v(t)) &= \frac{1}{2} \left( K(u(t),v(t)) + M_\mu(u(t),v(t)) \right) - P(u(t),v(t)) = E_\mu(u_0,v_0), \tag{Energy}
\end{align*}
where
\begin{align}
M_\mu(f,g) &:= \|f\|^2_{L^2(\R^3)} + \mu \|g\|^2_{L^2(\R^3)}, \label{defi-M-mu} \\
K(f,g) &:= \|\nabla f\|^2_{L^2(\R^3)} + \|\nabla g\|^2_{L^2(\R^3)}, \label{defi-K} \\
P(f,g) &:= \int_{\R^3} \frac{1}{36} |f(x)|^4 + \frac{9}{4} |g(x)|^4 + |f(x)|^2 |g(x)|^2 + \frac{1}{9} \rea \left( \overline{f}^3(x) g(x)\right) dx. \label{defi-P}
\end{align}

\noindent It is worth introducing since now the Pohozaev functional  
\begin{equation}  \label{defi-G}
G(f,g):= K(f,g) - 3 P(f,g),
\end{equation}
and, for later purposes, we rewrite the functionals  $P$ (see \eqref{defi-P}) by means of its density: namely
\[
P(f,g)=\int_{\R^{3}}N(f(x),g(x))dx
\]
where
\begin{equation}\label{defi-N}
N(f(x),g(x)):= \frac{1}{36} |f(x)|^4 + \frac{9}{4} |g(x)|^4 + |f(x)|^2 |g(x)|^2 + \frac{1}{9} \rea\(\overline{f}^3(x) g(x)\).
\end{equation}
The previous conservation laws can be formally proved by usual integration by part, then a rigorous  justification of them can be done by a classical regularization argument, see \cite{Cazenave}.\\

In order to introduce other invariance of the equations, let us give the following definition.
\begin{definition}
We say that the initial-value problem \eqref{SNLS} satisfies the mass-resonance condition provided that $\gamma=3.$
\end{definition} 
\noindent For $\gamma=3,$ \eqref{SNLS} has the Galilean invariance: namely, if $(u,v)$ is a solution to \eqref{SNLS}, then 
\begin{equation}\label{gal-trans}
u_\xi(t,x):= e^{ix\cdot\xi}e^{-t|\xi|^{2}i} u(t, x-2t\xi), \quad v_\xi(t,x):= e^{3ix\cdot\xi}e^{-3t|\xi|^{2}i}v(t, x-2t\xi)),
\quad \xi\in \R^{3},
\end{equation}
is also a solution to \eqref{SNLS} with initial data
$(e^{ix\cdot\xi}u_{0}, e^{3ix\cdot\xi}v_{0}).$

\begin{remark}\label{rmk:mass-res}
Notice that if $\gamma\neq3$, the system \eqref{SNLS} is not invariant under the Galilean transformations as in \eqref{gal-trans}. 
\end{remark}

As, in this paper, we are interested in long time behavior of solutions to \eqref{SNLS}, let us recall the notion of scattering. 
\begin{definition}
We say that a global solution $(u(t), v(t))$ to \eqref{SNLS} scatters in $H^{1}(\R^{3})\times H^{1}(\R^{3})$ if there exists a scattering state $(u_{\pm},v_{\pm})\in H^1(\R^3)\times H^(\R^3)$ such that
\begin{equation}\label{def:scattering}
\lim_{t\to\pm\infty}\|(u(t),v(t))-(\Sc_{1}(t)u_{\pm},\Sc_{2}(t) v_{\pm})\|_{H^{1}(\R^{3})\times H^{1}(\R^{3})}=0,
\end{equation}
where 
\begin{equation}\label{def:propagators}
\Sc_{1}(t)=e^{it (\Delta-1)} \hbox{ \quad and \quad  }\Sc_{2}(t)=e^{i\frac{t }{\gamma}(\Delta-\mu)}
\end{equation}
are linear Schr\"odinger propagators. 
\end{definition}
Note that the set of initial data such that solutions to \eqref{SNLS} satisfy \eqref{def:scattering} is non-empty, as solutions corresponding to small $H^1(\R^3)\times H^1(\R^3)$-data do scatter (see Section \ref{sec:pre}). \\
As already mention above, it is well-known that the dynamics of nonlinear Schr\"odinger-type equations is strongly related to the notion of ground states. Hence, we recall some basic facts about ground state standing waves related to \eqref{SNLS}. By standing waves, we mean solutions to \eqref{SNLS} of the form 
\[
(u(t,x), v(t,x)) = \left(e^{i\omega t} f(x), e^{3 i\omega t} g(x)\right),
\]
where $\omega \in \R$ is a frequency and $(f,g)$ is a real-valued solution to the system of elliptic equations
\begin{align} \label{syst-elli}
\left\{
\begin{aligned}
\Delta f - (\omega+1) f + \left(\frac{1}{9} f^2 + 2g^2 \right) f + \frac{1}{3} f^2 g &=0,\\
\Delta g - (\mu + 3\gamma \omega) g +(9g^2 + 2f^2) g +\frac{1}{9} f^3&=0.
\end{aligned}
\right.	
\end{align}
It was proved by Oliveira and  Pastor, see \cite{OP}, that solutions to \eqref{syst-elli} exist, provided that 
\begin{align} \label{cond-omega}
\omega > -\min \left\{ 1, \frac{\mu}{3\gamma}\right\}.
\end{align}
Moreover, a non-trivial solution $(\phi, \psi)$ to \eqref{syst-elli} is called ground state related to \eqref{syst-elli} if it minimizes the action functional 
\begin{equation}\label{Afu}
S_{\omega, \mu,\gamma} (f,g) := E_\mu(f,g) + \frac{\omega}{2} M_{3\gamma}(f,g),
\end{equation}
over all non-trivial solutions to \eqref{syst-elli}. Under the assumption \eqref{cond-omega}, the set of ground states related to \eqref{syst-elli} denoted by
\[
\Gc(\omega, \mu,\gamma):= \left\{ (\phi,\psi) \in \Ac_{\omega, \mu, \gamma}  \ : \ S_{\omega, \mu,\gamma}(\phi,\psi) \leq S_{\omega,\mu,\gamma}(f,g), \,\forall (f,g) \in \Ac_{\omega, \mu, \gamma}\right\}
\]
is not empty, where $\Ac_{\omega,\mu, \gamma}$ is the set of all non-trivial solutions to \eqref{syst-elli}. In particular, $\Gc(0, 3\gamma, \gamma) \ne \emptyset$.

It was shown (see \cite[Theorem 3.10]{OP}) that if $(u_0,v_0) \in H^1(\R^3) \times H^1(\R^3)$ satisfies
\begin{align} 
E_{\mu}(u_0,v_0) M_{3\gamma}(u_0,v_0) &< \frac{1}{2}E_{3\gamma}(\phi,\psi) M_{3\gamma}(\phi,\psi), \label{cond-ener} \\
K(u_0,v_0) M_{3\gamma}(u_0,v_0) &< K(\phi,\psi) M_{3\gamma}(\phi,\psi), \label{cond-gwp}
\end{align}
where $(\phi,\psi) \in \Gc(0,3\gamma,\gamma)$, then the corresponding solution to \eqref{SNLS} exists globally in time. The proof of this result is based on a continuity argument and the following sharp Gagliardo-Nirenberg inequality
\begin{align} \label{GN-ineq}
P(f,g) \leq C_{\opt} \left(K(f,g)\right)^{\frac{3}{2}} \left(M_{3\gamma}(f,g)\right)^{\frac{1}{2}}, \quad \forall (f,g) \in H^1(\R^3) \times H^1(\R^3).
\end{align}
This type of Gagliardo-Nirenberg inequality was established in \cite[Lemma 3.5]{OP}. Note that in \cite{OP}, this inequality was proved for real-valued $H^1$-functions. However, we can state it for complex-valued $H^1$-functions as well since $P(f,g) \leq P(|f|,|g|)$ and $\|\nabla(|f|)\|_{L^2(\R^3)} \leq \|\nabla f\|_{L^2(\R^3)}$. \\

We are now in position to state our first main result. The following theorem provides sufficient conditions to have scattering of solutions. More precisely, for data belonging to the set  given by conditions \eqref{cond-ener} and \eqref{cond-gwp}, solutions to \eqref{SNLS} satisfy \eqref{def:scattering}, for some scattering state $(u^\pm,v^\pm)$.

\begin{theorem}\label{Th1}
Let $\mu, \gamma>0$, and $(\phi,\psi) \in \Gc(0, 3\gamma, \gamma)$. Let $(u(t),v(t))$ the corresponding solution of \eqref{SNLS} with initial data $(u_{0}, v_{0})\in H^{1}(\R^{3})\times H^{1}(\R^{3})$. Assume that the initial data satisfies \eqref{cond-ener} and \eqref{cond-gwp}.
Provided that
\begin{itemize}[leftmargin=5mm]
\item (non-radial case) either $|\gamma-3|<\eta$ for some $\eta=\eta(E_{3\gamma}((u_{0}, v_{0})), M_{3\gamma}((u_{0}, v_{0})))>0$ small enough,
\item (radial case)  or $(u_{0}, v_{0})$ is radial,
\end{itemize}
then the solution of \eqref{SNLS} is global and scatters in $H^{1}(\R^{3})\times H^{1}(\R^{3})$. 
\end{theorem}
Our proof of the scattering results is based on the recent works by Dodson and Murphy \cite{DM-MRL} (for non-radial solutions) and \cite{DM-PAMS} (for radial solutions), using suitable scattering criteria and Morawetz-type estimates. In the non-radial case, we make use of an interaction Morawetz estimate to derive a space-time estimate. In the radial case, we make use of localized Morawetz estimates and radial Sobolev embeddings to show a suitable space-time bound of the solution.

Let us highlight the main novelties of this paper, regarding the linear asymptotic dynamics. For the classical focusing cubic equation in $H^1(\R^3),$ scattering (and blow-up) below the mass-energy threshold, was proved by Holmer and Roudenko in \cite{HR} for radial solutions, by exploiting the concentration/compactness and rigidity scheme in the spirit of Kenig and Merle, see \cite{KM}. The latter scattering result has been then extended to non-radial solution in Duyckaerts, Holmer, and Roudenko \cite{DHR}. To remove the radiality assumption, a crucial role is played by the invariance of the cubic NLS under the Galilean boost, which enables to have a zero momentum for the soliton-like solution. As observed in Remark \ref{rmk:mass-res}, equation \eqref{SNLS} lacks the Galilean invariance unless $\gamma=3.$ Hence we cannot rely on a Kenig and Merle road map to achieve our scattering results, and we instead build our analysis on the recent method developed by Dodson and Murphy, see \cite{DM-PAMS, DM-MRL}. In the latter two cited works, Dodson and Murphy give alternative proofs of the scattering results contained in \cite{HR, DHR}, which avoid the use of the concentration/compactness and rigidity method. They give a shorter proofs, though quite technical, based on Morawetz-type estimates. In our work, by borrowing from \cite{DM-PAMS, DM-MRL}, we prove  interaction Morawetz and Morawetz estimates for \eqref{SNLS}, and we prove Theorem \ref{Th1} for non-radial solutions which do not fit the mass-resonance condition, as well as for radially symmetric solutions. In this latter case, instead, we only need (localized) Morawetz estimates, which are less involved with respect to the interaction Morawetz ones, as we can take advantage of the spatial decay of radial Sobolev functions. \\

Our second main result is about formation of singularities in finite time for solutions to  \eqref{SNLS}. We state it for two classes of initial data. Indeed, besides the fact that these initial data must satisfy the a-priori bounds given by \eqref{cond-ener} and \eqref{cond-blow}  -- the latter (see below) replacing the condition \eqref{cond-gwp} yielding to global well-posedness -- they can belong either to the space of radial function, or to the anisotropic space of cylindrical function having finite variance in the last variable. The Theorem reads as follows. 

\begin{theorem} \label{theo-blow}
	Let $\mu, \gamma>0$, and $(\phi,\psi) \in \Gc(0, 3\gamma, \gamma)$. Let $(u_0,v_0) \in H^1(\R^3) \times H^1(\R^3)$ satisfy either $E_\mu(u_0,v_0)<0$ or, if $E_\mu(u_0,v_0) \geq 0$, we assume moreover that  \eqref{cond-ener} holds and 
	\begin{align} \label{cond-blow}
	K(u_0,v_0) M_{3\gamma}(u_0,v_0) > K(\phi,\psi) M_{3\gamma}(\phi,\psi).
	\end{align}
	If the initial data satisfy
	\begin{itemize}[leftmargin=5mm]
		\item either $(u_0,v_0)$ is radially symmetric,
		\item or $(u_0,v_0) \in \Sigma_3 \times \Sigma_3$,
		where
		\[
		\Sigma_3:= \left\{ f \in H^1(\R^3) \ : \ f(y,z) = f(|y|,z), z f \in L^2(\R^3) \right\}
		\]
		with $x=(y,z), y=(x_1,x_2) \in \R^2$ and $z \in \R,$
	\end{itemize}
	then the corresponding solution to \eqref{SNLS} blows-up in finite time.
\end{theorem}

Let us now comment previous known results about blow-up for \eqref{SNLS} and the one stated above, and highlight the main novelties of this paper regarding the blow-up achievements with respect to the previous literature.\\

In the mass-resonance case, i.e., $\gamma=3$, and provided $\mu=3\gamma=9$, the existence of finite time blow-up solutions to \eqref{SNLS} with finite variance initial data was proved in \cite[Theorems 4.6 and 4.8]{OP}. More precisely, they proved that if $(u_0,v_0) \in \Sigma (\R^3)\times \Sigma(\R^3)$ with $\Sigma(\R^3) = H^1(\R^3) \cap L^2(\R^3, |x|^2 dx)$ satisfying either $E_9(u_0,v_0) <0$ or if $E_9(u_0,v_0) \geq 0$, they moreover assumed that 
	\begin{align*}
	E_9(u_0,v_0) M_{9}(u_0,v_0) &< \frac{1}{2}E_{9}(\phi,\psi) M_{9}(\phi,\psi),\\
	K(u_0,v_0) M_{9}(u_0,v_0) &> K(\phi,\psi) M_{9}(\phi,\psi),
	\end{align*}
	where $(\phi,\psi) \in \Gc(0,9,3)$, then the corresponding solution to \eqref{SNLS} blows-up in finite time. The proof of the blow-up result in \cite{OP} is based on the following virial identity (see Remark \ref{rem-viri-iden})
	\begin{equation}\label{eq:gla}
	\frac{d^2}{dt^2} V(t) = 4 G(u(t),v(t)),
	\end{equation}
	where 
	\[
	V(t):= \int |x|^2 \left( |u(t,x)|^2 + 9 |v(t,x)|^2\right) dx.
	\]

	Using \eqref{eq:gla}, the finite time blow-up result follows from a convexity argument. For the power-type NLS equation, this kind of convexity strategy goes back to the early work of Glassey, see \cite{Glassey}, for finite variance solutions with negative initial energy. See the works by Ogawa and Tsutsumi \cite{OT} for the removal of the finiteness hypothesis of the variance, but with the addition of the radial assumption. See the already mentioned paper \cite{HR} for an extension to the cubic NLS up to the mass-energy threshold, of the results by Glassey, and Ogawa and Tsutsumi. 

If we do not assume the mass-resonance condition, or we do not assume that $\mu\neq3\gamma,$ the identity \eqref{eq:gla} ceases to be valid. Thus the convexity argument is no-more applicable in our general setting. The proof of Theorem \ref{theo-blow} above relies instead on an ODE argument, in the same spirit of our previous work \cite{DF}, using localized virial estimates and the negativity property of the Pohozaev functional (see Lemma \ref{lem-nega-G}). We point-out that  our result not only extends the one in \cite{OP} to radial and cylindrical solutions, but also extends it to the whole range of $\mu, \gamma>0$. It worth mentioning that blow-up in a full generality, i.e. for infinite-variance solutions with no symmetric assumptions, is still an open problem even for the classical cubic NLS. \\
	
We conclude this introduction by reporting some notation used along the paper, and by disclosing how the paper is organized.

\subsection{Notations}
We use the notation $X\lesssim Y$ to denote $X\leq C Y$ for some constant $C>0$. When  $X\lesssim Y$ and $Y\lesssim X$ (possibly for two different universal constants), we write $X\sim Y,$ or equivalently, we use the `big O'  notation $\Oc$, e.g., $X=\Oc(Y)$.  For $I\subset \R$ an interval, we denote the mixed norm
\[  
\|f\|_{L^{q}_{t}L^{r}_{x}(I\times\R^{3})}= \(\int_I \( \int_{\R^3} |f(t,x)|^r dx\)^{\frac{q}{r}} dt\)^{\frac{1}{q}}
\]
with the usual modifications when either $r$ or $q$ are infinity. When $q=r$, we simply write $\|f\|_{L^{q}_{t,x}(I\times\R^{3})}$. Let $f, g \in L^q_t L^r_x(I \times \R^3)$, we denote
\[
\|(f,g)\|_{L^q_t L^r_x \times L^q_t L^r_x(I\times \R^3)} := \|f\|_{L^{q}_{t}L^{r}_{x}(I\times\R^{3})} + \|g\|_{L^{q}_{t}L^{r}_{x}(I\times\R^{3})}
\]
and if $q=r$, we simply write
\[
\|(f,g)\|_{L^q_{t,x}\times L^q_{t,x}(I\times \R^3)} := \|f\|_{L^q_{t,x}(I\times\R^{3})} + \|g\|_{L^{q}_{t,x}(I\times\R^{3})}.
\]
The $L^p(\R^3)$ spaces, with $1\leq,p\leq\infty,$  are the usual Lebesgue spaces, as well as  spaces $W^{k,p}(\R^3)$ spaces, and their homogeneous versions, are the classical Sobolev spaces.
To lighten the notation along the paper, we will avoid to write $\R^3$  (unless necessary), as we are dealing with a three-dimensional problem. 

\subsection{Structure of the paper} This paper is organized as follows. In Section \ref{sec:pre}, we state  preliminary results that will be needed throughout the paper, and we will prove some coercivity  conditions which play a vital role to get the scattering results. In Section \ref{sec:VM}, we introduce localized quantities, and we derive localized virial estimates, Morawetz and interaction Morawetz estimates which will be the fundamental tools to establish the main results. The latter a-priori estimates will be shown in both radial and non-radial settings. In Section \ref{sec:sct}, we give scattering criteria for radial and non-radial solutions. We eventually prove, in Section \ref{sec:proofs-main}, the scattering results and the blow-up results, by employing the tools developed in the previous Sections. We conclude with the Appendixes \ref{sec:app:A} and \ref{sec:app:B}, devoted to the proofs of some results used along the paper. 

\section{Preliminary tools}\label{sec:pre}
In this section, we introduce some basic tools towards the proof of our main achievements. Specifically, we give a small data scattering result, as well as useful properties related to the ground states. We postpone the proof of some of the following results to the Appendix \ref{sec:app:A}.
\subsection{Small data theory}

We have the following small data scattering result, which will be useful in the sequel.

\begin{lemma}\label{lem-smal-scat}
Let $\mu,\gamma>0$, and $T>0$. Suppose that $(u,v)$ is a global $H^1$-solution to \eqref{SNLS} satisfying 
\[ 
\sup_{t\in \R}\|(u(t),v(t))\|_{H^{1}\times H^{1}}\leq E
\]
for some constant $E>0$. There exists $\epsilon_{\sd}=\epsilon_{\sd}(E)>0$ such that if 
\begin{equation}\label{Small-sc}
\| (\Sc_{1}(t-T)u(T),\Sc_{2}(t-T)v(T)) \|_{L_{t}^{4}L^{6}_{x}\times L_{t}^{4}L^{6}_{x}([T,\infty)\times \R^{3})}<\epsilon_{\sd},
\end{equation}
then the solution scatters forward in time.
\end{lemma}
\begin{proof}
See Appendix \ref{sec:app:A}.
\end{proof}

\subsection{Variational analysis}\label{Vari-Anal}
We first recall some basic properties of ground states in $\Gc(0,3\gamma,\gamma)$ and then show a coercivity condition (see \eqref{coer-prop}), which play a vital role to get scattering results. 

It was shown in \cite[Lemma 3.5]{OP} that any ground state $(\phi,\psi) \in \mathcal G(0,3\gamma,\gamma)$ optimizes the Gagliardo-Nirenberg inequality \eqref{GN-ineq}, that is
\[
C_{\opt} = \frac{P(\phi,\psi)}{\left(K(\phi,\psi)\right)^{\frac{3}{2}} \left(M_{3\gamma}(\phi,\psi)\right)^{\frac{1}{2}}}.
\]
Using the Pohozaev identities (see \cite[Lemma 3.4]{OP})
\begin{align} \label{poho-iden}
P(\phi,\psi) = S_{0,3\gamma,\gamma}(\phi,\psi) = E_{3\gamma}(\phi,\psi) = M_{3\gamma}(\phi,\psi)= \frac{1}{3} K(\phi,\psi),
\end{align}
we have
\begin{align} \label{opti-cons}
C_{\opt} = \frac{1}{3} \left( K(\phi,\psi) M_{3\gamma}(\phi,\psi)\right)^{-\frac{1}{2}}.
\end{align}

To employ some Morawetz estimates in the proof of the scattering theorem,  we will also use the following refined Gagliardo-Nirenberg inequality.

\begin{lemma} \label{lem-refi-GN-ineq}
Let $(\phi,\psi)\in \Gc(0, 3\gamma, \gamma)$. For any $(f,g)\in H^{1}\times H^{1}$ and $\xi_{1}$, $\xi_{2}\in \R^{3}$,
we have
\begin{equation}\label{refi-GN-ineq}
P(|f|,|g|)\leq \frac{1}{3}\(\frac{K(f,g) M_{3\gamma}(f,g)}{K(\phi,\psi) M_{3\gamma}(\phi,\psi)}\)^{\frac{1}{2}}K(e^{ix\cdot\xi_{1}}f,e^{ix\cdot\xi_{2}}g).
\end{equation}
\end{lemma}
\begin{proof}
See Appendix \ref{sec:app:A}.
\end{proof}
We conclude this preliminary section by giving the following two coercivity results. 

\begin{lemma}\label{L22}
Let $\mu, \gamma>0$, and $(\phi,\psi)\in \Gc(0, 3\gamma, \gamma)$. Let $(u_{0},v_{0})\in H^{1}\times H^{1}$ satisfy \eqref{cond-ener} and \eqref{cond-gwp}. Then the corresponding solution to \eqref{SNLS} exists globally in time and satisfies
\begin{equation}\label{est-K}
\sup_{t\in \R} K(u(t),v(t))\leq 6E_{\mu}(u_{0},v_{0}).
\end{equation}
Moreover, there exists $\delta=\delta(u_0,v_0,\phi,\psi)>0$ such that 
\begin{equation}\label{coer-1}
K(u(t),v(t))M_{3\gamma}(u(t),v(t)) \leq (1-\delta)K(\phi,\psi)M_{3\gamma}(\phi,\psi)
\end{equation}
for all $t\in \R$.
\end{lemma}
\begin{proof}
See Appendix \ref{sec:app:A}.
\end{proof}

\begin{lemma} \label{lem-coer-2}
Let $\mu,\gamma>0$, and $(\phi,\psi) \in \mathcal G(0,3\gamma,\gamma)$. Let $(u_{0},v_{0})\in H^{1}\times H^{1}$ satisfy \eqref{cond-ener} and \eqref{cond-gwp}. Let $\delta$ be as in \eqref{coer-1}. Then there exists $R=R(\delta, u_0,v_0, \phi,\psi)>0$ sufficiency large such that for any $z\in \R^{3}$,
\begin{equation}\label{Cb2}
\begin{aligned}
K\(\Gamma_{R}(\cdot-z)u(t),\Gamma_{R}(\cdot-z)v(t)\) &M_{3\gamma}\(\Gamma_{R}(\cdot-z)u(t),\Gamma_{R}(\cdot-z)v(t)\) \\
&\leq \(1-\frac{\delta}{2}\)K(\phi,\psi)M_{3\gamma}(\phi,\psi)
\end{aligned}
\end{equation}
uniformly for $t\in \R$, where $\Gamma_{R}(x):=\Gamma\(\frac{x}{R}\)$ with $\Gamma$ a cutoff function satisfying $0\leq \Gamma(x)\leq 1$ for all $x\in \R^3$. Moreover, there exists $\nu=\nu(\delta) >0$ independent on $t$ so that
for any $\xi_{1}, \xi_{2}\in \R^{3}$, and any $z\in \R^3$,
\begin{equation}\label{coer-prop}
\begin{aligned}
K\(\Gamma_{R}(\cdot-z)e^{ix\cdot\xi_{1}}u(t),\Gamma_{R}(\cdot-z)e^{ix\cdot\xi_{2}}v(t)\) &- 3 P\(\Gamma_{R}(\cdot-z)u(t),\Gamma_{R}(\cdot-z)v(t)\) \\
&\geq \nu  K\(\Gamma_{R}(\cdot-z)e^{ix\cdot\xi_{1}}u(t),\Gamma_{R}(\cdot-z)e^{ix\cdot\xi_{2}}v(t)\)
\end{aligned}
\end{equation}
for any $t\in \R$.
\end{lemma}
\begin{proof}
See Appendix \ref{sec:app:A}.
\end{proof}

\section{Virial and Morawetz estimates}\label{sec:VM}
This section is devoted to the proof of  virial-type, Morawetz-type, and interaction Morawetz-type estimates, which will be crucial for the proof of the main Theorems \ref{Th1} and \ref{theo-blow}.

\subsection{Virial estimates}
We start with  the following identities.  In what follows we use the Einstein convention, so repeated indices are summed. 
\begin{lemma} \label{Imporide}
	Let $\mu, \beta, \gamma>0$, and $(u,v)$ be a $H^1$-solution to \eqref{SNLS}. Then the following identities hold:
	\begin{align}\label{Idg}
	\partial_{t}(|u|^{2}+\gamma\beta|v|^{2})&=-2\nabla\cdot\IM (\overline{u}\nabla u)
	-2\beta\nabla\cdot\IM (\overline{v}\nabla v)+\frac{2}{3}\(1-\frac{\beta}{3}\)\IM (u^{3}\overline{v}),\\
	\label{Idn}
	\partial_{t}\IM (\overline{u}\partial_k u+\gamma\overline{v} \partial_k v)&= 
	\frac{1}{2}\partial_{k}\Delta (|u|^{2}+|v|^{2}) -2\partial_{j}\RE(\partial_j\overline{u} \partial_k u+\partial_j\overline{v} \partial_kv)+2\partial_{k}N(u,v),
	\end{align}
	where $N$ is as in \eqref{defi-N}.	In particular, we have
	\begin{align*}
	\partial_{t}(|u|^{2}+\gamma^{2}|v|^{2})&=-2\nabla\cdot\IM (\overline{u}\nabla u)
	-2\gamma\nabla\cdot\IM (\overline{v}\nabla v)+\frac{2}{3}\(1-\frac{\gamma}{3}\)\IM (u^{3}\overline{v}),\\
	\partial_{t}(|u|^{2}+3\gamma|v|^{2})&=-2\nabla\cdot\IM (\overline{u}\nabla u)
	-6\nabla\cdot\IM (\overline{v}\nabla v).
	\end{align*}
\end{lemma}
\begin{proof}
See Appendix \ref{sec:app:B}.
\end{proof}

A direct consequence of Lemma \ref{Imporide} is the following localized virial identity related to \eqref{SNLS}.

\begin{lemma} \label{lem-viri-iden}
	Let $\mu,\gamma >0$, and $\varphi: \R^3 \rightarrow \R$ be a sufficiently smooth and decaying function. Let $(u,v)$ be a $H^1$-solution to \eqref{SNLS} defined on the maximal time interval $(-T_-,T_+)$. Define
	\begin{align} \label{defi-M-varphi}
	\Mcal_\varphi(t):= 2 \ima \int \nabla \varphi(x) \cdot (\nabla u\overline{u}+\gamma\nabla v \overline{v})(t,x) dx.
	\end{align}
	Then we have for all $t\in (-T_-,T_+)$, 
	\begin{align*} 
	\frac{d}{dt}\Mcal_\varphi(t)&=- \int \Delta^2 \varphi(x) ( |u|^2 + |v|^2 )(t,x) dx + 4\rea\int \partial^2_{jk}\varphi(x)(\partial_j\overline{u} \partial_k u+ \partial_j\overline{v} \partial_k v)(t,x)dx  \\
	&-4\int \Delta \varphi(x)N(u,v)(t,x) dx.
	\end{align*}
\end{lemma}
The following Corollary is easy to get.
\begin{corollary} \label{rem-viri-iden}
Recall the definition of $G,N,P$ in \eqref{defi-G}, \eqref{defi-N}, and \eqref{defi-P}, respectively.
\begin{itemize}[leftmargin=5mm]
		\item[(i)] If $\varphi(x) = |x|^2$, 	
		\begin{equation}\label{eq:variance}
		\frac{d}{dt} \Mcal_{|x|^2}(t) = 8 G(u(t),v(t)).
		\end{equation}
		\item[(ii)] If $\varphi$ is radially symmetric, by denoting $|x|=r,$ we have		
		\begin{equation}\label{cor:ii}
		\begin{aligned}
		\frac{d}{dt} \Mcal_\varphi(t) &= -\int \Delta^2 \varphi(x) (|u|^2 +  |v|^2)(t,x) dx + 4\int \frac{\varphi'(r)}{r} (|\nabla u|^2 + |\nabla v|^2 )(t,x) dx \\
		 &+ 4 \int \left(\frac{\varphi''(r)}{r^2} - \frac{\varphi'(r)}{r^3} \right) (|x\cdot \nabla u|^2 +  |x\cdot \nabla  v|^2 )(t,x) dx \\
		&-4\int \Delta \varphi(x) N(u,v)(t,x)dx. 
		\end{aligned}
		\end{equation}
		\item[(iii)] If $\varphi$ is radial and $(u,v)$ is also radial, then 
		\begin{equation}\label{cor:iii}
		\begin{aligned}
		\frac{d}{dt} \Mcal_\varphi(t) &= -\int \Delta^2 \varphi(x) (|u|^2 + |v|^2)(t,x) dx  + 4 \int \varphi''(r) (|\nabla u|^2 + |\nabla v|^2)(t,x) dx \\
		&- 4\int \Delta \varphi(x) N(u,v)(t,x)dx. 
		\end{aligned}
		\end{equation}
		\item[(iv)] Denote $x=(y,z)$ with $y=(x_1, x_2) \in \R^2$ and $z \in \R$. Let $\psi: \R^2 \rightarrow \R$ be a sufficiently smooth and decaying function. Set $\varphi(x) = \psi(y) + z^2$. If $(u(t),v(t)) \in \Sigma_3 \times \Sigma_3$ for all $t\in (-T_-,T_+)$, then we have
		\begin{equation}\label{cor:iv}
		\begin{aligned}
		\frac{d}{dt} \Mcal_\varphi(t) &= -\int \Delta^2_y \psi(y) (|u|^2 +  |v|^2)(t,x) dx + 4\int \psi''(\rho) (|\nabla_y u|^2 +  |\nabla_y v|^2)(t,x) dx \\
		& + 8 \left(\|\partial_z u(t)\|^2_{L^2} + \|\partial_z v(t)\|^2_{L^2}\right) - 8 P(u(t),v(t))  -4\int \Delta_y \psi(y) N(u,v)(t,x)dx,
		\end{aligned}
		\end{equation}
		where $\rho = |y|.$ 
	\end{itemize}
\end{corollary}
\begin{proof}
See Appendix \ref{sec:app:B}.
\end{proof}
We now aim to construct precise localization functions that we will use to get the desired main results of the paper. Let $\zeta: [0,\infty) \rightarrow [0,2]$ be a smooth function satisfying
\[
\zeta(r):= \left\{
\begin{array}{ccl}
2 &\text{if}& 0 \leq r \leq 1, \\
0 &\text{if}& r\geq 2.
\end{array}
\right.
\]
We define the function $\vartheta:[0,\infty) \rightarrow [0, \infty)$ by
\begin{align} \label{defi-vartheta}
\vartheta(r):= \int_0^r \int_0^\tau \zeta(s) ds d\tau.
\end{align}
For $R>0$, we define the radial function $\varphi_R: \R^3 \rightarrow \R$ by
\begin{align} \label{defi-varphi-R}
\varphi_R(x)=\varphi_R(r):= R^2 \vartheta(r/R), \quad r=|x|.
\end{align}
We readily check that, $ \forall x \in \R^3$ and $\forall r\geq 0,$
\[
2\geq \varphi''_R(r) \geq 0, \quad 2-\frac{\varphi'_R(r)}{r} \geq 0, \quad 6-\Delta \varphi_R(x) \geq 0.  
\]
We are ready to state the first virial estimate for radially symmetric solutions. 
 
\begin{lemma} \label{lem-viri-est-rad}
	Let $\mu,\gamma >0$. Let $(u,v)$ be a radial $H^1$-solution to \eqref{SNLS} defined on the maximal time interval $(-T_-,T_+)$. Let $\varphi_R$ be as in \eqref{defi-varphi-R} and denote $\Mcal_{\varphi_R}(t)$ as in \eqref{defi-M-varphi}. Then we have for all $t\in (-T_-,T_+)$,
	\begin{align} \label{viri-est-rad}
	\frac{d}{dt} \Mcal_{\varphi_R}(t) \leq 4 G(u(t),v(t))  + C R^{-2} K(u(t),v(t)) + CR^{-2}
	\end{align}
	for some constant $C>0$ depending only on $\mu, \gamma$, and $M_{3\gamma}(u_0,v_0)$, where $G$ is as in \eqref{defi-G}.
\end{lemma}

\begin{proof}
	By \eqref{cor:iii}, we have for all $t\in (-T_-,T_+)$,
	\begin{align*}
	\frac{d}{dt} \Mcal_{\varphi_R}(t) &= -\int \Delta^2 \varphi_R(x) (|u|^2 + |v|^2)(t,x) dx  + 4 \int_{\R^3} \varphi''_R(r) (|\nabla u|^2 + |\nabla v|^2)(t,x) dx \\
	&-4\int \Delta \varphi_R(x) N(u,v)(t,x)dx. 
	\end{align*}
	We rewrite, using $G-K+3P=0,$
	\begin{align*}
	\frac{d}{dt} \Mcal_{\varphi_R}(t)&= 8 G(u(t),v(t)) - 8 K(u(t),v(t)) + 24 P(u(t),v(t)) \\
	&-\int \Delta^2 \varphi_R(x) (|u|^2 + |v|^2)(t,x) dx + 4 \int \varphi''_R(r) (|\nabla u|^2 + |\nabla v|^2)(t,x) dx \\
	& -4\int \Delta \varphi_R(x) N(u,v)(t,x)dx \\
	&= 8 G(u(t),v(t)) -\int \Delta^2 \varphi_R(x) (|u|^2 + |v|^2)(t,x) dx \\
	& - 4 \int (2-\varphi''_R(r)) (|\nabla u|^2 + |\nabla v|^2)(t,x) dx  + 4 \int (6-\Delta \varphi_R(x)) N(u,v)(t,x)dx.
	\end{align*}
	As $\|\Delta^2 \varphi_R\|_{L^\infty} \lesssim R^{-2}$, the conservation of mass implies that
	\[
	\left| \int_{\R^3} \Delta^2 \varphi_R(x) (|u|^2 +|v|^2)(t,x) dx \right| \lesssim R^{-2}.
	\]
	The latter, together with $\varphi''_R (r) \leq 2$ for all $r\geq0$, $\|\Delta \varphi_R\|_{L^\infty} \lesssim 1$, $\varphi_R(x) = |x|^2$ on $|x| \leq R$, and H\"older's inequality, yield
	\[
	\frac{d}{dt}\Mcal_{\varphi_R}(t) \leq 8 G(u(t),v(t)) + CR^{-2} + C\int_{|x|\geq R} |u(t,x)|^4 + |v(t,x)|^4dx,
	\]
	where we have used the fact that (see \eqref{defi-N})
	\[
	|N(u,v)| \lesssim |u|^4 + |v|^4.
	\]
	To estimate the last term, we recall the following radial Sobolev embedding (see e.g., \cite{CO}): for a  radial function $f\in H^1(\R^3)$, we have  
	\begin{align} \label{rad-sobo}
	\sup_{x \ne 0} |x| |f(x)| \leq C\|\nabla f\|^{\frac{1}{2}}_{L^2} \|f\|^{\frac{1}{2}}_{L^2}.
	\end{align}
	Thanks to \eqref{rad-sobo} and the conservation of mass, we estimate
	\begin{align*}
	\int_{|x|\geq R} |u(t,x)|^4 dx &\leq \sup_{|x|\geq R} |u(t,x)|^2 \|u(t)\|^2_{L^2}\\
	&\lesssim R^{-2} \sup_{|x|\geq R} \left(|x||u(t,x)|\right)^2 \|u(t)\|^2_{L^2}\\
	&\lesssim R^{-2} \|\nabla u(t)\|_{L^2} \|u(t)\|^3_{L^2} \\
	&\lesssim R^{-2} \|\nabla u(t)\|_{L^2} \\
	&\lesssim R^{-2} \left(\|\nabla u(t)\|^2_{L^2} +1\right).
	\end{align*}
	It follows that
	\[
	\frac{d}{dt} \Mcal_{\varphi_R}(t) \leq 8 G(u(t),v(t)) + CR^{-2} + CR^{-2} \left(\|\nabla u(t)\|^2_{L^2}+\|\nabla v(t)\|^2_{L^2}\right).
	\]
	The proof is complete.
\end{proof}

Next we derive localized virial estimates for cylindrically symmetric solutions (we also mention here \cite{BF20, BFG20, Mar, DF, Inui1, Inui2}, for the qualitative analysis of dispersive-type equations in anisotropic spaces). 
To this end, we introduce
\begin{align} \label{defi-psi-R}
\psi_R(y)= \psi_R(\rho) := R^2 \zeta(\rho/R), \quad \rho =|y|
\end{align}
and set
\begin{align} \label{defi-varphi-R-psi}
\varphi_R(x):= \psi_R(y) + z^2.
\end{align}

\begin{lemma} \label{lem-viri-est-cyli}
	Let $\mu,\gamma>0$. Let $(u,v)$ be a $\Sigma_3$-solution to \eqref{SNLS} defined on the maximal time interval $(-T_-,T_+)$. Let $\varphi_R$ be as in \eqref{defi-varphi-R-psi} and denote $\Mcal_{\varphi_R}(t)$ as in \eqref{defi-M-varphi}. Then we have for all $t\in (-T_-,T_+)$,
	\begin{align} \label{viri-est-cyli}
	\frac{d}{dt} \Mcal_{\varphi_R}(t) \leq 8 G(u(t),v(t)) + CR^{-1} K(u(t),v(t)) + CR^{-2}
	\end{align}
	for some constant $C>0$ depending only on $\mu,\gamma$, and $M(u_0,v_0)$.
\end{lemma}

\begin{proof}
	By \eqref{cor:iv}, we have for all $t\in (-T_-,T_+)$,
	\begin{align*}
	\frac{d}{dt} \Mcal_{\varphi_R}(t) &= -\int \Delta^2_y \psi_R(y) (|u|^2 +  |v|^2)(t,x) dx + 4\int_{\R^3} \psi''_R(\rho) (|\nabla_y u|^2 +  |\nabla_y v|^2)(t,x) dx \\
	& + 8 \left(\|\partial_z u(t)\|^2_{L^2} + \|\partial_z v(t)\|^2_{L^2}\right) - 8 P(u(t),v(t))  -4\int \Delta_y \psi_R(y) N(u,v)(t,x)dx,
	\end{align*}
	where $\rho=|y|$. It follows that
	\begin{align*}
	\frac{d}{dt} \Mcal_{\varphi_R}(t) &\leq 8 G(u(t),v(t)) +CR^{-2} - 4 \int (2-\psi''_R(\rho)) ( |\nabla_y u|^2 +  |\nabla_y v|^2)(t,x) dx \\
	&+4 \rea \int (4-\Delta_y\psi_R(y)) N(u,v)(t,x)dx.
	\end{align*}
	As $\psi''_R(\rho) \leq 2$ and $\|\Delta_y \psi_R\|_{L^\infty} \lesssim 1$, the H\"older's inequality implies that
	\begin{align} \label{est-cyli}
	\frac{d}{dt} \Mcal_{\varphi_R}(t) \leq 8 G(u(t),v(t)) + CR^{-2} + C \int_{|y|\geq R} |u(t,x)|^4 + |v(t,x)|^4 dx.
	\end{align}
	We estimate
	\begin{align*}
	\int_{|y|\geq R} |u(t,x)|^4 dx &\leq \int_{\R} \|u(t,z)\|^2_{L^2_y} \|u(t,z)\|^2_{L^\infty_y(|y|\geq R)} dz \\
	&\leq \sup_{z \in \R} \|u(t,z)\|^2_{L^2_y} \left(\int_{\R} \|u(t,z)\|^2_{L^\infty_y(|y|\geq R)} dz \right).
	\end{align*}
	Set $g(z):= \|u(t,z)\|^2_{L^2_y}$, we have
	\begin{align*}
	g(z) = \int_{-\infty}^{z} \partial_s g(s) ds &= 2 \int_{-\infty}^{z} \rea \int_{\R^2} \overline{u}(t,y,s) \partial_s u(t,y,s) dy  ds \\
	&\leq 2 \|u(t)\|_{L^2_x} \|\partial_z u(t)\|_{L^2_x}
	\end{align*}
	which, by the conservation of mass, implies that
	\begin{align} \label{est-cyli-1}
	\sup_{z \in \R} \|u(t,z)\|^2_{L^2_y} \lesssim \|\partial_z u(t)\|_{L^2_x}.
	\end{align}
	By the radial Sobolev embedding \eqref{rad-sobo} with respect to the $y$-variable, we have
	\begin{align}
	\int \|u(t,z)\|^2_{L^\infty_y(|y|\geq R)} dz &\lesssim R^{-1}  \int  \|\nabla_y u(t,z)\|_{L^2_y} \|u(t,z)\|_{L^2_y} dz \nonumber \\
	&\lesssim R^{-1} \left( \int  \|\nabla_y u(t,z)\|^2_{L^2_y} dz\right)^{1/2} \left( \int \|u(t,z)\|^2_{L^2_y} dz\right)^{1/2} \nonumber \\
	&\lesssim R^{-1} \|\nabla_y u(t)\|_{L^2_x} \|u(t)\|_{L^2_x} \nonumber \\
	&\lesssim R^{-1} \|\nabla_y u(t)\|_{L^2_x}. \label{est-cyli-2}
	\end{align}
	Collecting \eqref{est-cyli-1} and \eqref{est-cyli-2}, we get
	\begin{align*}
	\int_{|y|\geq R} |u(t,x)|^4 dx &\lesssim R^{-1} \|\nabla_y u(t)\|_{L^2_x} \|\partial_z u(t)\|_{L^2_x} \\
	&\lesssim R^{-1} \left(\|\nabla_y u(t)\|^2_{L^2_x} + \|\partial_z u(t)\|^2_{L^2_x}\right) \\
	&\lesssim R^{-1} \|\nabla u(t)\|^2_{L^2_x}.
	\end{align*}
	The latter and \eqref{est-cyli} give \eqref{viri-est-cyli}. The proof is complete.
\end{proof}
\subsection{Interaction Morawetz estimates. Non-radial setting}
Following \cite{WY}, let $\chi$ be a decreasing radial smooth function such that $\chi(x)=1$ for $|x|\leq 1-\sigma$, $\chi(x)=0$ for $|x|\geq 1$,  and $|\nabla \chi| \lesssim \sigma^{-1}$, where $0<\sigma<1$ is a small constant.

Let $R>1$ be a large parameter. We define the following radial functions 
\begin{align*}
\Phi_R(x) &=\frac{1}{\omega_{3}R^{3}}\int\chi_{R}^{2}(x-z)\chi_{R}^{2}(z)dz, \\
\Phi_{1,R}(x,y)&=\frac{1}{\omega_{3}R^{3}}\int\chi_{R}^{2}({x-z})\chi_{R}^{4}({y-z})dz,
\end{align*}
where $\chi_{R}(x):=\chi\(\frac{x}{R}\)$ and $\omega_{3}$ is the volume of unit ball in $\R^{3}$. We also define the functions 
\[
\Psi_R(x)=\frac{1}{|x|}\int^{|x|}_{0}\Phi_R(r)dr, \quad \Theta_R(x)=\int^{|x|}_{0} r \Psi_R(r)dr.
\]

We collect below some properties of the above functions.
\begin{remark}[\cite{DM-MRL}] \label{Rema1} 
Straightforward calculations give: 
\begin{itemize}[leftmargin=5mm]
\item the identities  $\partial_{j}\Theta_R(x)=x_j\Psi_R(x)$ and $\partial_{j}\Psi_R(x)=\frac{x_{j}}{|x|^{2}}(\Phi_R(x)-\Psi_R(x))$, and in particular,
\begin{align} \label{prop-cutoff-1}
\Delta \Theta_R(x)=2\Psi_R(x)+\Phi_R(x), \quad \partial^2_{jk}\Theta_R(x)=\delta_{jk}\Phi_R(x)+P_{jk}(x)(\Psi_R(x)-\Phi_R(x)), 
\end{align}
where $P_{jk}(x)=\delta_{jk}-\frac{x_{j}x_{k}}{|x|^{2}}$ with $\delta_{jk}$ the Kronecker symbol;
\item and that the estimates below are satisfied:
\begin{equation}\label{prop-cutoff-2}
\begin{aligned}
&\Psi_R(x)-\Phi_R(x)\geq 0,\quad\qquad |\Psi_R(x)|\lesssim \min\left\{1, \frac{R}{|x|}\right\}, \\
& |\nabla \Phi_R(x)|\lesssim \frac{1}{\sigma R}, \,\,\,\quad\quad\qquad
|\nabla \Psi_R(x)|\lesssim \frac{1}{\sigma}\min \left\{\frac{1}{R}, \frac{R}{|x|^2}\right\}, \\
& |\Phi_R(x)-\Phi_{1,R}(x)|\lesssim \sigma, \quad\quad |\Psi_R(x)-\Phi_R(x)|\lesssim \frac{1}{\sigma}\min\left\{\frac{|x|}{R}, \frac{R}{|x|}\right\}.
\end{aligned}
\end{equation}
\end{itemize}
\end{remark}

Let $(u,v)$ be a global $H^1$-solution to \eqref{SNLS} with initial data $(u_0,v_0)$ satisfying \eqref{cond-ener} and \eqref{cond-gwp}. We define the interaction Morawetz quantity adapted to system \eqref{SNLS} by
\[
\Mcal^{\otimes 2}_{R}(t)=2\iint L_{\gamma}(u,v)(t,y)\nabla \Theta_R(x-y)\cdot \IM (\overline{u}\nabla u+\gamma \overline{v}\nabla v)(t,x)dxdy,
\]
where
\[
L_{\gamma}(u,v)(t, x):=(|u|^{2}+\gamma^{2}|v|^{2})(t,x).
\]
From the conservation of mass, \eqref{est-K}, and \eqref{prop-cutoff-2}, we have
\[ 
\sup_{t\in\R}|\Mcal^{\otimes 2}_{R}(t)| \lesssim R.
\]
By Lemma \ref{Imporide}, we have
\begin{equation}\label{dtm}
\partial_{t}L_{\gamma}(u,v)=-2\nabla\cdot\IM (\overline{u}\nabla u)
-2\gamma\nabla\cdot\IM (\overline{v}\nabla v)+\frac{2}{3}\(1-\frac{\gamma}{3}\)\IM (u^{3}\overline{v})
\end{equation}
and
\begin{align*}
\partial_{t}\IM (\overline{u} \partial_k u+\gamma\overline{v} \partial_kv)&=
-2\partial_{j}\RE(\partial_j\overline{u} \partial_k u+ \partial_j\overline{v} \partial_k v)+
\frac{1}{2}\partial_{k}\Delta (|u|^{2}+|v|^{2})+2\partial_{k}N(u,v),
\end{align*}
where we recall that
\[
N(u,v)=\frac{1}{36}|u|^{4}+\frac{9}{4}|v|^{4}+|u|^{2}|v|^{2}+\frac{1}{9}\RE (\overline{u}^{3}v).
\]
Here repeated indices are summed. Moreover, by using integration by parts, we readily see that
\begin{align}\label{Ln}
\frac{d}{dt} \Mcal^{\otimes 2}_{R}(t)=&4\iint L_{\gamma}(u,v)(t,y) \nabla \Theta_R(x-y)\cdot \nabla N(u,v)(t,x)dxdy
\\  \label{Ln1}
&+\iint L_{\gamma}(u,v)(t,y) \nabla \Theta_R(x-y)\cdot \nabla\Delta (|u|^{2}+|v|^{2})(t,x)dxdy
\\\label{Ln2}
&-4\iint L_{\gamma}(u,v)(t,y) \partial_{k} \Theta_R(x-y)\partial_{j}\RE(\partial_j\overline{u} \partial_k u+\partial_j \overline{v} \partial_kv)(t,x)dxdy
\\ \label{Ln3}
&+2\iint \partial_{t}L_{\gamma}(u,v)(t,y)\nabla \Theta_R(x-y)\cdot \IM (\overline{u}\nabla u+\gamma \overline{v}\nabla v)(t,x)dxdy.
\end{align}

We are able to prove the following interaction Morawetz estimates, which will play a fundamental role for the proof of the scattering theorem in the non-radial framework.
\begin{proposition}\label{Imnn}
Let $\mu,\gamma>0$, and $(\phi,\psi) \in \Gc(0,3\gamma,\gamma)$. Let $(u_{0},v_{0}) \in H^1\times H^1$ satisfy \eqref{cond-ener} and \eqref{cond-gwp}. Let $(u,v)$ be the corresponding global solution to \eqref{SNLS}. Then for arbitrary small $\epsilon>0$, there exist $T_{0}=T_{0}(\epsilon)$, $J=J(\epsilon)$, $R_{0}=R_{0}(\epsilon, u_0,v_0 \phi,\psi)$ sufficiently large and $\sigma=\sigma(\epsilon)$,
$\eta=\eta(\epsilon)$ sufficiently small such that if $|\gamma-3|<\eta$, then for any $a\in \R$, 
\begin{multline}\label{Pl11}
\frac{1}{JT_{0}}\int^{a+T_{0}}_{a}\int^{R_{0}e^{J}}_{R_{0}}\frac{1}{R^{3}}\int_{\R^{3}}
W_{\gamma}(\chi_{R}(\cdot-z)u(t),\chi_{R}(\cdot-z)v(t)) \\
\times K(\chi_{R}(\cdot-z)u^{\xi}(t),\chi_{R}(\cdot-z)v^{\xi}(t))dz\frac{dR}{R}dt \lesssim \epsilon,
\end{multline}
where $(u^{\xi}(t,x), v^{\xi}(t,x)):= (e^{ix\cdot\xi}u(t,x), e^{i\gamma x\cdot\xi}v(t,x))$ for some $\xi=\xi(t,z,R)\in \R^{3}$ and
\[
W_{\gamma}(f,g)=\int_{\R^{3}}L_{\gamma}(f,g)(x)dx.
\]
\end{proposition}

\begin{proof}
Since $\Delta \Theta_R(x-y)=3\Phi_{1,R}(x,y)+3(\Phi_R-\Phi_{1,R})(x,y)+2(\Psi_R-\Phi_R)(x-y)$, by integration by parts, we have
\begin{align}\label{S1}
\eqref{Ln}=&-12\iint L_{\gamma}(u,v)(t,y) \Phi_{1,R}(x-y)N(u,v)(t,x)dxdy
\\\label{S2}
&-12\iint L_{\gamma}(u,v)(t,y) (\Phi_R-\Phi_{1,R})(x-y)N(u,v)(t,x)dxdy
\\\label{S3}
&-8\iint L_{\gamma}(u,v)(t,y) (\Psi_R-\Phi_R)(x-y)N(u,v)(t,x)dxdy.
\end{align}
Again, by integration by parts and Remark \ref{Rema1}, we have
\begin{align} \label{Ln11}
\eqref{Ln1}=\iint  L_{\gamma}(u,v)(t,y) \nabla (3\Phi_R(x-y)+2(\Psi_R-\Phi_R)(x-y))\cdot \nabla(|u|^{2}+|v|^{2})(t,x)dxdy.
\end{align}
We will treat \eqref{S2}, \eqref{S3}, and \eqref{Ln11} as error terms. Moreover, by Remark \ref{Rema1}, we get
\begin{align}\label{R1}
\eqref{Ln2}&= 4\iint L_{\gamma}(u,v)(t,y) \Phi_R(x-y)(|\nabla u|^{2}+|\nabla v|^{2})(t,x)dxdy\\\label{R2}
&+4\iint L_{\gamma}(u,v)(t,y) (\Psi_R-\Phi_R)(x-y)P_{jk}(x-y)\RE(\partial_j\overline{u} \partial_k u+ \partial_j\overline{v} \partial_k v)(t,x)dxdy.
\end{align}
Similarly, by \eqref{dtm} and Remark \ref{Rema1}, we see that
\begin{align}\label{J1}
	\eqref{Ln3}&=-4\iint \Phi_R(x-y) \IM (\overline{u}\nabla u+\gamma \overline{v}\nabla v)(t,y)\cdot
	\IM (\overline{u}\nabla u+\gamma \overline{v}\nabla v)(t,x)dxdy\\\label{J2}
	&-4\iint (\Psi_R-\Phi_R)(x-y)P_{jk}(x-y)\IM(\overline{u} \partial_ku+\gamma\overline{v} \partial_k v)(t,y)
	\IM(\overline{u} \partial_j u +\gamma\overline{v} \partial_j v)(t,y)dxdy\\\label{J3}
	&+\frac{4}{3}\(1-\frac{\gamma}{3}\)\iint \nabla \Theta_R(x-y)\cdot \IM(\overline{u}\nabla u+\gamma \overline{v}\nabla v)(t,x)\IM(u^{3}\overline{v})(t,y)dxdy.
	\end{align}
Now, let $\slashed\nabla_y$ denote the angular derivative centered at $y$, namely
\[
\De_y f(x) := \nabla f(x) - \frac{x-y}{|x-y|} \(\frac{x-y}{|x-y|} \nabla f(x)\)
\]
and similarly for $\De_x$. We have
\begin{equation}\label{Pds}
\begin{aligned}
\eqref{R2}+\eqref{J2} &=4\iint (\Psi_R-\Phi_R)(x-y)\Big((|\De_{y}u|^{2}+|\De_{y}v|^{2})(t,x) (|u|^{2}+\gamma^{2}|v|^{2})(t,y)\\
& -\IM (\overline{u}\De_{y}u+\gamma \overline{v}\De_{y}v)(t,x)\cdot
\IM (\overline{u}\De_{x}u+\gamma \overline{v}\De_{x}v)(t,y)\Big)dxdy.
\end{aligned}
\end{equation}
Hence $\psi_R-\phi_R$ is radial and non-negative, by the Cauchy-Schwarz inequality, we infer that 
\[
\eqref{R2}+\eqref{J2}=\eqref{Pds}\geq 0.
\]
On the other hand,  as $\chi_{R}$  is radial and non-negative, we have
\begin{align}\nonumber
\eqref{R1}+\eqref{J1} &=\frac{4}{\omega_{3}R^{3}}\iiint \chi^2_{R}(x-z)\chi^2_{R}(y-z)\Big(
(|\nabla u|^{2}+|\nabla v|^{2})(t,x) (|u|^{2}+\gamma^{2}| v|^{2})(t,y) \nonumber \\
& - \IM (\overline{u}\nabla u+\gamma \overline{v}\nabla v)(t,y)\cdot
	\IM (\overline{u}\nabla u+\gamma \overline{v}\nabla v)(t,x)\Big) dxdydz \nonumber\\	
&=\frac{4}{\omega_{3}R^{3}}\int B(u,v)(t,z)dz, \label{Gin}
\end{align}
where
\begin{align*}
B(u,v)(t,z):&=\int \chi^{2}_{R}(x-z)(|\nabla u|^{2}+|\nabla v|^{2})(t,x)dx
\int \chi^{2}_{R}(y-z)(|u|^{2}+\gamma^{2}|v|^{2})(t,y)dy\\
& -\left|\int \chi^{2}_{R}(x-z)\IM (\overline{u}\nabla u+\gamma \overline{v}\nabla v)(t,x)dx\right|^{2}.
\end{align*}
Notice that $B(u,v)$ is invariant under the gauge transformation
\[(u(t,x), v(t,x))\mapsto (u^{\xi}(t,x), v^{\xi}(t,x)):= (e^{ix\cdot\xi}u(t,x), e^{i\gamma x\cdot\xi}v(t,x))\]
for any $\xi\in \R^{3}$. Indeed, we see that
\begin{align*}
L_{\gamma}(u^{\xi},v^{\xi})&= L_{\gamma}(u,v), \quad V_{\gamma}(u^{\xi},v^{\xi})=\xi L_{\gamma}(u,v)+V_{\gamma}(u,v),\\
H(u^{\xi},v^{\xi})&=|\xi|^{2} L_{\gamma}(u,v)+H(u,v)+2\xi\cdot V_{\gamma}(u,v),
\end{align*}
where
\[
V_{\gamma}(u,v)(t,x):=\IM (\overline{u}\nabla u+\gamma \overline{v}\nabla v)(t,x), \quad
H(u,v)(t,x):= (|\nabla u|^{2}+|\nabla v|^{2})(t,x),
\]
which implies that $B(u^{\xi},v^{\xi})=B(u,v)$. Next, we define
\[
\xi(t,z,R):=-\frac{\mathlarger{\int}\chi^{2}_{R}(x-z)V_{\gamma}(u,v)(t,x)dx}
{\mathlarger{\int}\chi^{2}_{R}(x-z)L_{\gamma}(u,v)(t,x)dx}
\]
provided that the denominator is non-zero; otherwise we can define $\xi(t,z,R)\equiv 0$. With this choice of $\xi$, we have
\[ 
\int \chi^{2}_{R}(x-z)V_{\gamma}(u^{\xi},v^{\xi})(t,x)dx=0.
\]
Combining this with \eqref{Gin}, we infer that
\begin{align*}
\eqref{R1}+\eqref{J1}=\frac{4}{\omega_{3}R^{3}}\int \(\int \chi^{2}_{R}(x-z)H(u^{\xi},v^{\xi})(t,x)dx
 \int \chi^{2}_{R}(y-z)L_{\gamma}(u,v)(t,y)dy\)dz.
\end{align*}
Therefore, by the above identity, \eqref{S1}, \eqref{S2}, \eqref{Ln11}, and \eqref{J3}, we get
\begin{align}
\frac{4}{\omega_{3}R^{3}}&\int_{\R^{3}}\(\int \chi^{2}_{R}(y-z)L_{\gamma}(u,v)(t,y)dy\) \nonumber \\
&\times \(\int \chi^{2}_{R}(x-z)H(u^{\xi},v^{\xi})(t,x)-3\chi^{4}_{R}(x-z)N(u,v)(t,x) dx\)dz \label{Mc1} \\
&\leq \frac{d}{dt} \Mcal^{\otimes 2}_{R}(t) \label{Mc2}\\
\label{Mc3}
&+\iint  L_{\gamma}(u,v)(t,y) (12(\Phi_R-\Phi_{1,R})+8(\Psi_R-\Phi_R))(x-y) N(u,v)(t,x)dxdy\\\label{Mc4}
&-\iint  L_{\gamma}(u,v)(t,y) (3 \nabla\Phi_R+ 2 \nabla(\Psi_R-\Phi_R))(x-y) 
\cdot \nabla(|u|^2+|v|^2)(t,x)dxdy\\\label{Mc5}
&+\frac{4}{3}\(\frac{\gamma}{3}-1\)\iint  \nabla \Theta_R(x-y)\cdot \IM(\overline{u}\nabla u+\gamma \overline{v}\nabla v)(t,x)\IM(u^{3}\overline{v})(t,y)dxdy.
\end{align}
Now, we consider \eqref{Mc1}. Since
\[
\int |\nabla (\chi f)|^{2}dx=\int \chi^{2}|\nabla f|^{2}dx-\int\chi\Delta\chi |f|^{2}dx,
\]
we get
\begin{equation}\label{Sust}
\begin{aligned}
\int \chi^{2}_{R}(x-z)H(u^{\xi},v^{\xi})(t,x)dx&=
\int H(\chi_{R}(\cdot-z)u^{\xi},\chi_{R}(\cdot-z)v^{\xi})(t,x)dx\\
&+\int \chi_{R}(x-z)\Delta\(\chi_{R}(x-z)\)(|u|^{2}+|v|^{2})(t,x)dx.
\end{aligned}
\end{equation}
Thus, substituting \eqref{Sust} in \eqref{Mc1} and using Lemma \ref{lem-coer-2} with $\chi_R$ instead of $\Gamma_R$, we see that there exists $\nu >0$ such that  
\begin{align*}
\frac{1}{JT_{0}}\int^{a+T_{0}}_{a}\int^{R_{0}e^{J}}_{R_{0}}\eqref{Mc1}\frac{dR}{R}dt &\geq \frac{4\nu}{\omega_{3}JT_{0}}\int^{a+T_{0}}_{a}\int^{R_{0}e^{J}}_{R_{0}}\frac{1}{R^{3}}\int_{\R^{3}}
\bigg( W_{\gamma}(\chi_{R}(\cdot-z)u(t),\chi_{R}(\cdot-z)v(t)) \\
&\mathrel{\phantom{\quad\quad\int^{a+T_{0}}_{a}\int^{R_{0}e^{J}}_{R_{0}}}} \times K(\chi_{R}(\cdot-z)u^{\xi}(t),\chi_{R}(\cdot-z)v^{\xi}(t))dz\bigg)\frac{dR}{R}dt\\
&+\frac{4\nu}{\omega_{3}JT_{0}}\int^{a+T_{0}}_{a}\int^{R_{0}e^{J}}_{R_{0}}\frac{1}{R^{3}}\int_{\R^{3}}
W_{\gamma}(\chi_{R}(\cdot-z)u(t),  \chi_{R}(\cdot-z)v(t))\\
&\mathrel{\phantom{\quad}} \times \(\int_{\R^{3}}\chi_{R}(\cdot-z)\Delta\( \chi_{R}(\cdot-z)\)(|u|^{2}+|v|^{2})(t,x)dx  \)dz\frac{dR}{R}dt.
\end{align*}
By the conservation of mass and the fact that $\|\Delta(\chi_R)\|_{L^\infty} \lesssim R^{-2}$, the absolute value of the second term in the right hand side can be bounded by
\[
\frac{4\nu}{\omega_3 J T_0} \int_a^{a+T_0} \int_{R_0}^{R_0e^J} CR^{-2}\frac{dR}{R} dt \lesssim \frac{1}{JR_0^2}.
\]
This implies that
\begin{align}\nonumber
\frac{1}{JT_{0}}\int^{a+T_{0}}_{a}\int^{R_{0}e^{J}}_{R_{0}}&\frac{1}{R^{3}}\int_{\R^{3}}
W_{\gamma}( \chi_{R}(\cdot-z)u(t),\chi_{R}(\cdot-z)v(t))K(\chi_{R}(\cdot-z)u^{\xi}(t),\chi_{R}(\cdot-z)v^{\xi}(t))dz\frac{dR}{R}dt\\\label{Plcsa}
&\lesssim \frac{1}{JT_{0}}\int^{a+T_{0}}_{a}\int^{R_{0}e^{J}}_{R_{0}}\eqref{Mc1}\frac{dR}{R}dt+\frac{1}{JR_{0}^{2}}.
\end{align}
Next, as $|\Mcal^{\otimes 2}_{R}(t)|\lesssim R$, we have 
\begin{equation}\label{Cul2}
\left|\frac{1}{JT_{0}}\int^{a+T_{0}}_{a}\int^{R_{0}e^{J}}_{R_{0}} \eqref{Mc2}\frac{dR}{R} dt\right| \leq \frac{1}{JT_0} \int^{R_{0}e^{J}}_{R_{0}} \sup_{t\in [a, a+T_0]} |\Mcal^{\otimes 2}_R(t)| \frac{dR}{R} \lesssim \frac{R_{0}e^{J}}{JT_0}.
\end{equation}
By \eqref{prop-cutoff-2}, the conservation of mass, \eqref{est-K}, and Sobolev embedding, we have
\begin{align*}
\Big| \frac{1}{JT_0} \int_a^{a+T_0} &\int_{R_0}^{R_0e^J} \iint  L_\gamma(u,v)(t,y) (\Phi_R-\Phi_{1,R})(x-y) N(u,v)(t,x) dx dy \frac{dR}{R} dt \Big| \\
&\lesssim \frac{1}{JT_0} \int_a^{a+T_0} \int_{R_0}^{R_0e^J} \sigma \frac{dR}{R} dt \lesssim \sigma,
\end{align*}
where we have used the fact that
\begin{align*}
\int  |L_\gamma(u,v)(t,y)|dy &\lesssim M_{3\gamma}(u(t),v(t)), \\ 
\int |N(u,v)(t,x)|dx &\lesssim \|(u(t),v(t))\|^4_{L^4 \times L^4 } \lesssim \|(u(t),v(t))\|^4_{H^1 \times H^1 }.
\end{align*}
Using \eqref{prop-cutoff-2}, we see that
\begin{align*}
\Big| \frac{1}{JT_0} &\int_a^{a+T_0} \int_{R_0}^{R_0e^J} \iint  L_\gamma(u,v)(t,y) (\Psi_R-\Phi_R)(x-y) N(u,v)(t,x) dxdy \frac{dR}{R}dt \Big| \\
&\lesssim \frac{1}{ \sigma JT_0} \int_a^{a+T_0} \int_{R_0}^{R_0e^J} \iint  |L_\gamma(u,v)(t,y)| \min \left\{\frac{|x-y|}{R},\frac{R}{|x-y|} \right\} |N(u,v)(t,x)|dxdy \frac{dR}{R} dt \\
&\lesssim \frac{1}{ \sigma JT_0} \int_a^{a+T_0} \iint  |L_\gamma(u,v)(t,y)| \left( \int_{R_0}^{R_0e^J} \min \left\{\frac{|x-y|}{R},\frac{R}{|x-y|} \right\} \frac{dR}{R} \right) |N(u,v)(t,x)|dxdy dt \\
&\lesssim \frac{1}{ \sigma J},
\end{align*}
where we have used the fact that 
\[
\int_0^\infty \min \left\{\frac{|x-y|}{R},\frac{R}{|x-y|} \right\} \frac{dR}{R} \lesssim 1.
\]
We thus get
\begin{equation}\label{Cul3}
\left|\frac{1}{JT_{0}}\int^{a+T_{0}}_{a}\int^{R_{0}e^{J}}_{R_{0}} \eqref{Mc3}\frac{dR}{R} dt\right| \lesssim \sigma +\frac{1}{ \sigma J}.
\end{equation}
As $|\nabla\Phi_R(x)|, |\nabla(\Psi_R-\Phi_R)(x)| \lesssim \frac{1}{ \sigma R}$, we see that
\begin{align} \label{Cul4}
\left|\frac{1}{JT_{0}}\int^{a+T_{0}}_{a}\int^{R_{0}e^{J}}_{R_{0}} \eqref{Mc4}\frac{dR}{R} dt\right| \lesssim \frac{1}{\sigma JR_0}.
\end{align}
Finally, as $|\gamma-3|<\eta$ and $|\nabla \Theta_R(x)| \lesssim R$, we infer from the conservation of mass, \eqref{est-K}, and Sobolev embedding that
\begin{equation}\label{bnms}
\left|\frac{1}{JT_{0}}\int^{a+T_{0}}_{a}\int^{R_{0}e^{J}}_{R_{0}}\eqref{Mc5} \frac{dR}{R}dt\right|\lesssim  \frac{\eta}{JT_{0}}\int^{a+T_{0}}_{a}\int^{R_{0}e^{J}}_{R_{0}}dRdt \lesssim \eta \frac{R_{0}e^{J}}{J}.
\end{equation}
Combining these estimates \eqref{Plcsa}, \eqref{Cul2}, \eqref{Cul3}, \eqref{Cul4}, and \eqref{bnms}, we obtain
\begin{align*}
\frac{1}{JT_{0}}\int^{a+T_{0}}_{a}\int^{R_{0}e^{J}}_{R_{0}}&\frac{1}{R^{3}}\int_{\R^{3}}
W_{\gamma}( \chi_{R}(\cdot-z)u(t),\chi_{R}(\cdot-z)v(t))K(\chi_{R}(\cdot-z)u^{\xi}(t),\chi_{R}(\cdot-z)v^{\xi}(t))dz\frac{dR}{R}dt\\
&\lesssim \frac{1}{JR_{0}^{2}}+\frac{R_{0}e^{J}}{JT_0}+\sigma+\frac{1}{\sigma J}+\frac{1}{ \sigma JR_{0}}+\eta \frac{R_{0}e^{J}}{J},
\end{align*}
which shows \eqref{Pl11} by choosing $\sigma=\epsilon, J=\epsilon^{-3}$, $R_{0}=\epsilon^{-1}$, $T_{0}=e^{\epsilon^{-3}}$,
and $\eta=e^{-\epsilon^{-3}}$. The proof is complete.
\end{proof}

\subsection{Morawetz estimates. Radial setting}
We now turn our attention to the proof of the radial version of the Morawetz estimate which will be essential in the proof of the scattering theorem in the radially symmetric setting. In this context, we take advantage of the radial Sobolev embedding to get some spatial decay.
\begin{lemma}\label{Mst1}
Let $\mu,\gamma>0$, and $(\phi,\psi)\in \Gc(0, 3\gamma, \gamma)$. Let $(u_0,v_0)\in H^1 \times H^1 $ be radially symmetric satisfying \eqref{cond-ener} and \eqref{cond-gwp}. Then for any $T>0$ and $R=R(u_0,v_0, \phi,\psi)>0$ sufficiently large, the corresponding global solution to \eqref{SNLS} satisfies
\begin{equation}\label{Mirs}
\frac{1}{T}\int^{T}_{0}\int_{|x|\leq \frac{R}{2}}\(|u(t,x)|^{\frac{10}{3}}+|v(t,x)|^{\frac{10}{3}}\)dxdt
\lesssim \frac{R}{T}+\frac{1}{R^{2}}.
\end{equation}
\end{lemma}

\begin{proof}
Let $\varphi_R$ be as in \eqref{defi-varphi-R} and define $\Mcal_{\varphi_R}(t)$ as in \eqref{defi-M-varphi}. By the Cauchy-Schwarz inequality, the conservation of mass, and \eqref{est-K}, we have
\begin{align} \label{est-M-varphi-R}
\sup_{t\in \R} |\Mcal_{\varphi_R}(t)| \lesssim R.
\end{align} 
By \eqref{cor:iii}, we have
\begin{align*}
\frac{d}{dt} \Mcal_{\varphi_R}(t) &= - \int \Delta^2 \varphi_R(x) (|u|^2+|v|^2)(t,x) dx + 4\int\varphi''_R(r) (|\nabla u|^2+|\nabla v|^2)(t,x)\\
&- 4 \int\Delta \varphi_R(x) N(u,v)(t,x) dx.
\end{align*}
As $\varphi_R(x)=|x|^2$ for $|x|\leq R$, we see that
\begin{align*}
\frac{d}{dt} \Mcal_{\varphi_R}(t) &= 8 \( \int_{|x|\leq R} (|\nabla u|^2+|\nabla v|^2)(t,x) dx - 3 \int_{|x|\leq R} N(u,v)(t,x) dx \) \\
&- \int \Delta^2\varphi_R(x) (|u|^2+|v|^2)(t,x) dx + 4 \rea \int_{|x| >R} \partial^2_{jk} \varphi_R(x) (\partial_j \overline{u} \partial_k u + \partial_j \overline{v} \partial_k v) (t,x) dx \\
&- 4 \int_{|x|>R} \Delta \varphi_R(x) N(u,v)(t,x) dx.
\end{align*}
Since $\|\Delta^2\varphi_R\|_{L^\infty } \lesssim R^{-2}$, the conservation of mass implies
\[
\int  \Delta^2 \varphi_R(x) (|u|^2+|v|^2)(t,x) dx \lesssim R^{-2}.
\]
As $(u,v)$ is radially symmetric, we use the fact
\[
\partial_j = \frac{x_j}{r} \partial_r, \quad\partial^2_{jk} =\(\frac{\delta_{jk}}{r} - \frac{x_j x_k}{r^3}\) \partial_r + \frac{x_j x_k}{r^2} \partial^2_r
\]
to get
\[
\partial^2_{jk} \varphi_R(x) \partial_j \overline{u}(t,x) \partial_k u(t,x) = \varphi''_R(r) |\partial_r u(t,r)|^2 \geq 0
\]
which implies 
\[
\rea \int_{|x|>R} \partial^2_{jk}\varphi_R(x) (\partial_j \overline{u}\partial_k u + \partial_j \overline{v} \partial_k v)(t,x) \geq 0.
\]
On the other hand, by arguing as in the proof of Lemma \ref{lem-viri-est-rad}, we have
\[
\left| \int_{|x|>R} \Delta \varphi_R(x) N(u,v)(t,x) dx \right| \lesssim R^{-2} K(u(t),v(t)) \lesssim R^{-2}.
\] 
Thus we get
\begin{align} \label{est-M-varphi-R-app}
\frac{d}{dt} \Mcal_{\varphi_R}(t) \geq 8 \( \int_{|x|\leq R} (|\nabla u|^2+|\nabla v|^2)(t,x) dx - 3 \int_{|x|\leq R} N(u,v)(t,x) dx \) + C R^{-2}
\end{align}
for all $t\in \R$. Now, let $\varrho_R(x)=\varrho(x/R)$ with $\varrho$ as in \eqref{defi-varrho}. We have
\begin{align*}
\int  |\nabla(\varrho_R u(t))|^2 dx &= \int \varrho^2_R |\nabla u(t)|^2 dx - \int \varrho_R \Delta\varrho_R |u(t)|^2 dx \\
&= \int_{|x|\leq R} |\nabla u(t)|^2 - \int_{R/2 \leq |x| \leq R} (1-\varrho^2_R) |\nabla u(t)|^2 dx - \int \varrho_R \Delta\varrho_R |u(t)|^2 dx
\end{align*}
and 
\[
\int N(\varrho_R u, \varrho_R v)(t,x) dx = \int_{|x| \leq R} N(u,v)(t,x) dx + \int_{R/2\leq |x| \leq R} \(N(\varrho_R u, \varrho_R v) - N(u,v) \)(t,x) dx.
\]
It follows that
\begin{align*}
\int_{|x|\leq R} &(|\nabla u|^2 +|\nabla v|^2)(t,x) dx - 3 \int_{|x|\leq R} N(u,v)(t,x) dx \\
&=\int (|\nabla(\varrho_R u)|^2 + |\nabla(\varrho_R v)|^2)(t,x) dx - 3\int  N(\varrho_R u, \varrho_R v)(t,x) dx \\
&+ \int_{R/2 \leq |x| \leq R} (1-\varrho_R^2(x)) (|\nabla u|^2 +|\nabla v|^2) (t,x) dx \\
&+ \int  \varrho_R(x) \Delta \varrho_R(x) (|u|^2+|v|^2) (t,x) dx - 3\int_{R/2 \leq |x| \leq R} \( N(\varrho_R u, \varrho_R v) - N(u,v)\) (t,x) dx.
\end{align*}
As $0 \leq \rho_R \leq 1$ and $\|\Delta \varrho_R\|_{L^\infty } \lesssim R^{-2}$, the conservation of mass, \eqref{est-K}, and the radial Sobolev embedding, we have
\begin{align*}
\int_{|x|\leq R} (|\nabla u|^2 +|\nabla v|^2)(t,x) dx &- 3 \int_{|x|\leq R} N(u,v)(t,x) dx \\
&\geq K(\varrho_R u(t), \varrho_R v(t))- 3 P(\varrho_R u(t), \varrho_R v(t)) + \Oc(R^{-2}).
\end{align*} 
Thanks to \eqref{coer-prop} with $\varrho_R$ in place of $\Gamma_R$ and $z=\xi_1= \xi_2 =0$, there exist $R=R(u_0,v_0,\phi,\psi)>0$ sufficiently large and $\nu=\nu(u_0,v_0,\phi,\psi)>0$ such that
\[
\int_{|x|\leq R} (|\nabla u|^2 +|\nabla v|^2)(t,x) dx - 3 \int_{|x|\leq R} N(u,v)(t,x) dx \geq \nu K(\varrho_R u(t), \varrho_R v(t)) + \Oc(R^{-2})
\]
for all $t\in \R$. This together with \eqref{est-M-varphi-R-app} yield
\[
\nu K(\varrho_R u(t), \varrho_R v(t)) \leq \frac{d}{dt} \Mcal_{\varphi_R}(t) + C R^{-2}
\]
for all $t\in \R$. Integrating on $[0, T]$ and using \eqref{est-M-varphi-R}, we get
\[
\frac{1}{T}\int^{T}_{0}K(\varrho_{R}u(t), \varrho_{R}v(t))dt\lesssim
\frac{R}{T}+\frac{1}{R^{2}}.
\]
In particular, we have
\[
\frac{1}{T}\int^{T}_{0}\|\nabla(\rho_{R}u(t))\|^{2}_{L^{2} }dt\lesssim
\frac{R}{T}+\frac{1}{R^{2}}
\]
which together with the Gagliardo-Nirenberg inequality 
\[
\|  u \|^{\frac{10}{3}}_{L^{\frac{10}{3}} }\lesssim \| \nabla u \|^{2}_{L^{2} } 
\|  u \|^{\frac{4}{3}}_{L^{2} } 
\]
imply
\[
\frac{1}{T}\int^{T}_{0}\|\rho_{R}u(t)\|^{\frac{10}{3}}_{L^{\frac{10}{3}} }dt  
\lesssim \frac{1}{T}\int^{T}_{0}\|\nabla(\rho_{R}u(t))\|^{2}_{L^{2} }dt\lesssim
\frac{R}{T}+\frac{1}{R^{2}}.
\]
By the choice of $\varrho_R$, we obtain
\[
\frac{1}{T}\int^{T}_{0}\int_{|x|\leq \frac{R}{2}}|u(t,x)|^{\frac{10}{3}}dxdt\lesssim
\frac{R}{T}+\frac{1}{R^{2}}.
\]
A similar estimate holds for $v$. The proof is complete.
\end{proof}

\section{Scattering criteria}\label{sec:sct}

In this section, we give scattering criteria for solution to \eqref{SNLS} in the  spirit of Dodson and Murphy \cite{DM-MRL, DM-PAMS} (see also \cite{WY}). Let us start with the scattering criterion for non-radial solutions.
\begin{proposition} \label{prop-scat-crit-non-rad}
Let $\mu, \gamma>0$. Suppose that $(u, v)$ is a global $H^1$-solution to \eqref{SNLS} satisfying 
\begin{equation}\label{Est1}
\sup_{t\in \R}\| (u(t),v(t))  \|_{H^{1} \times H^{1} }
\lesssim E
\end{equation}
for some constant $E>0$. Then there exist $\epsilon=\epsilon(E)>0$  small enough and $T_{0}=T_{0}(\epsilon, E)>0$  sufficiently large such that if for any $a\in \R$, there exists $t_0\in(a,a+T_{0})$ such that
\begin{equation}\label{Dtn}
\| (u(t),v(t))  \|_{L^{5}_{t,x}\times L^{5}_{t,x}([t_0-\epsilon^{-\frac{1}{4}}, t_0] \times \R^3)} \lesssim \epsilon,
\end{equation}
then the solution scatters forward in the time.
\end{proposition}

\begin{proof}
By Lemma \ref{lem-smal-scat}, it suffices to show that there exists $T>0$ such that 
\begin{equation}\label{Oc}
\| (\Sc_{1}(t-T)u(T),\Sc_{2}(t-T)v(T))  \|_{L_{t}^{4}L^{6}_{x}\times L_{t}^{4}L^{6}_{x}([T,\infty)\times \R^{3})} \lesssim \epsilon^{\frac{1}{32}}.
\end{equation} 
To prove \eqref{Oc}, we first write
\[
(\Sc_{1}(t-T)u(T),\Sc_{2}(t-T)v(T))=(\Sc_{1}(t)u_{0},\Sc_{2}(t)v_{0})+i \int_0^T (\Sc_1(t-s)F_1(s), \Sc_2(t-s)F_2(s)) ds.
\]
By Sobolev embedding, Strichartz estimates, and the monotone convergence theorem, there exists
$T_{1}>0$ sufficiently large such that if $T>T_{1}$, then
\begin{equation}\label{Cdt}
\|(\Sc_{1}(t)u_{0},\Sc_{2}(t)v_{0})\|_{L_{t}^{4}L^{6}_{x}\times L_{t}^{4}L^{6}_{x}([T,\infty)\times \R^{3})} \lesssim \epsilon.
\end{equation}
We take $a=T_{1}$ and $T=t_0$, where $a$ and $t_0$ are as in \eqref{Dtn}, we write
\[
i \int_0^T (\Sc_1(t-s)F_1(s), \Sc_2(t-s)F_2(s)) ds =: H_1(t) + H_2(t),
\]
where 
\[
H_{j}(t)=i\int_{I_{j}}(\Sc_{1}(t-s)F_{1}(s), \Sc_{2}(t-s)F_{2}(s))ds, \quad I_{1}=[0, T-\epsilon^{-\frac{1}{4}}], \quad 
I_{2}=[T-\epsilon^{-\frac{1}{4}}, T].
\]
To estimate $H_2$, we observe that
\begin{equation}\label{on1}
\|(u,v)\|_{L_{t}^{2}\dot{W}^{\frac{1}{2},6}_{x}\times L_{t}^{2}\dot{W}^{\frac{1}{2},6}_{x}([T-\epsilon^{-\frac{1}{4}}, T]\times \R^{3})}
\lesssim 1.
\end{equation}
Indeed, by Strichartz estimates, fractional chain rule, \eqref{Est1}, and \eqref{Dtn}, we have
\begin{align*}
\|(u,v)\|&_{L_{t}^{2}\dot{W}^{\frac{1}{2},6}_{x}\times L_{t}^{2}\dot{W}^{\frac{1}{2},6}_{x}([T-\epsilon^{-\frac{1}{4}}, T]\times \R^{3})}
\\
&\lesssim E+ \| (u,v) \|^{2}_{L^{5}_{t,x}\times L^{5}_{t,x}([T-\epsilon^{-\frac{1}{4}}, T] \times \R^3)}
\|(u,v)\|_{L_{t}^{2}\dot{W}^{\frac{1}{2},6}_{x}\times L_{t}^{2}\dot{W}^{\frac{1}{2},6}_{x}([T-\epsilon^{-\frac{1}{4}}, T]\times \R^{3})}\\
&\lesssim E+ \epsilon^2\|(u,v)\|_{L_{t}^{2}\dot{W}^{\frac{1}{2},6}_{x}\times L_{t}^{2}\dot{W}^{\frac{1}{2},6}_{x}([T-\epsilon^{-\frac{1}{4}}, T]\times \R^{3})}.
\end{align*}
By choosing  $\epsilon$ small enough, we get \eqref{on1}. By Sobolev embedding and Strichartz estimates, we see that
\begin{align*}
\|H_{2}\|_{L_{t}^{4}L^{6}_{x}\times L_{t}^{4}L^{6}_{x}([T,\infty)\times \R^{3})} \lesssim \|(u,v)\|^{2}_{L_{t,x}^{5}\times L_{t,x}^{5}([T-\epsilon^{-\frac{1}{4}}, T]\times \R^{3})}
\|(u,v)\|_{L_{t}^{2}\dot{W}^{\frac{1}{2},6}_{x}\times L_{t}^{2}\dot{W}^{\frac{1}{2},6}_{x}([T-\epsilon^{-\frac{1}{4}}, T]\times \R^{3})}
\end{align*}
which together with \eqref{Dtn} and \eqref{on1} imply
\begin{equation}\label{Pii}
\|H_{2}\|_{L_{t}^{4}L^{6}_{x}\times L_{t}^{4}L^{6}_{x}([T,\infty)\times \R^{3})}\lesssim \epsilon^2.
\end{equation}

On the other hand, we claim that
\begin{equation}\label{cla1}
\|H_{1}\|_{L_{t}^{4}L^{6}_{x}\times L_{t}^{4}L^{6}_{x}([T,\infty)\times \R^{3})}\lesssim \epsilon^{\frac{1}{32}}.
\end{equation}
In fact, we notice that 
\[
H_{1}(t)=(\Sc_{1}(t-T+\epsilon^{-\frac{1}{4}})u(T-\epsilon^{-\frac{1}{4}}), \Sc_{2}(t-T+\epsilon^{-\frac{1}{4}})u(T-\epsilon^{-\frac{1}{4}}))
-(\Sc_{1}(t)u_{0}, \Sc_{2}(t)v_{0})
\]
which, by Strichartz estimates, implies
\begin{equation*}
\|H_{1}\|_{L_{t}^{4}L^{3}_{x}\times L_{t}^{4}L^{3}_{x}([T,\infty)\times \R^{3})}
\lesssim \|(u(T-\epsilon^{-\frac{1}{4}}), v(T-\epsilon^{-\frac{1}{4}}))\|_{L^2\times L^2 } + \|(u_{0}, v_{0})\|_{L^2 \times L^2}\lesssim E.
\end{equation*}
Moreover, as
\[
\|(F_{1}(t), F_{2}(t))\|_{L^1 \times L^1 } \lesssim \|(u(t),v(t))\|^3_{L^3 \times L^3 } \lesssim \| (u(t),v(t)) \|^{3}_{H^{1}\times H^{1} }
\lesssim E^{3},
\] 
we have from the dispersive estimate \eqref{Dpe} and Young's inequality that
\begin{equation*}
\|H_{1}\|_{L_{t}^{4}L^{\infty}_{x}\times L_{t}^{4}L^{\infty}_{x}([T,\infty)\times \R^{3})}
\lesssim
\left\|\int^{T-\epsilon^{-\frac{1}{4}}}_{0}|t-s|^{-3/2} ds\right\|_{L^{4}_{t}([T,\infty))}
\lesssim \epsilon^{\frac{1}{16}}.
\end{equation*}
By interpolation, we get
\begin{align*}
\|H_{1}\|_{L_{t}^{4}L^{6}_{x}\times L_{t}^{4}L^{6}_{x}([T,\infty)\times \R^{3})} \leq \|H_{1}\|^{1/2}_{L_{t}^{4}L^{3}_{x}\times L_{t}^{4}L^{3}_{x}([T,\infty)\times \R^{3})}
\|H_{1}\|^{1/2}_{L_{t}^{4}L^{\infty}_{x}\times L_{t}^{4}L^{\infty}_{x}([T,\infty)\times \R^{3})}
\lesssim \epsilon^{\frac{1}{32}}
\end{align*}
which proves \eqref{cla1}. Collecting \eqref{Cdt}, \eqref{Pii}, and \eqref{cla1}, we obtain \eqref{Oc}, and the proof is complete.
\end{proof}

Let us give now an analogous of the previous Criterion in the radial setting. 

\begin{proposition}[Scattering criterion for radial solutions]\label{prop-scat-crit-rad}
Let $\mu,\gamma>0$. Suppose that $(u,v)$ is a global $H^1$-solution to \eqref{SNLS} satisfying
\begin{equation}\label{Lrs}
\sup_{t\in \R}\|(u(t),v(t))\|_{H^{1} \times H^{1}}\leq E
\end{equation}
for some constant $E>0$. Then there exist $\epsilon=\epsilon(E)>0$ and $R=R(E)>0$ such that if
\begin{equation}\label{Taos}
\liminf_{t\rightarrow \infty}\int_{|x|\leq R} \(|u(t,x)|^{2}+3\gamma|v(t,x)|^{2}\)dx\leq \epsilon^{2},
\end{equation}
then the solution scatters forward in time.
\end{proposition}

\begin{proof}
Let $\epsilon>0$ be a small constant. By Lemma \ref{lem-smal-scat}, it suffices to show the existence of $T=T(\epsilon)>0$ such that
\begin{align} \label{scat-crit-app}
\| (\Sc_{1}(t-T)u(T),\Sc_{2}(t-T)v(T)) \|_{L_{t}^{4}L^{6}_{x}\times L_{t}^{4}L^{6}_{x}([T,\infty)\times \R^{3})}<\epsilon^{\frac{1}{32}}.
\end{align}
To show this, we follow the argument of \cite[Lemma 2.2]{DM-PAMS}. By the Strichartz estimates and the monotone
convergence theorem, there exists $T=T(\epsilon)>0$ sufficiently large such that
\begin{equation}\label{Cddf}
\|(\Sc_{1}(t)u_{0},\Sc_{2}(t)v_{0})\|_{L_{t}^{4}L^{6}_{x}\times L_{t}^{4}L^{6}_{x}([T,\infty)\times \R^{3})}<\epsilon.
\end{equation}
As in the proof of Proposition \ref{prop-scat-crit-non-rad}, we write
\[
(\Sc_{1}(t-T)u(T),\Sc_{2}(t-T)v(T)) = (\Sc_{1}(t)u_{0},\Sc_{2}(t)v_{0})+H_{1}(t)+H_{2}(t),
\]
where 
\[
H_{j}(t)= i\int_{I_{j}}(\Sc_{1}(t-s)F_{1}(s), \Sc_{2}(t-s)F_{2}(s))ds, \quad I_{1}=[0, T-\epsilon^{-\frac{1}{4}}], \quad 
I_{2}=[T-\epsilon^{-\frac{1}{4}},T].
\]
By \eqref{Taos} and enlarging $T$ if necessary, we have
\begin{equation}\label{Tapro}
\int\varrho_{R}(x)\(|u(T,x)|^{2}+3\gamma|v(T,x)|^{2}\)dx\leq \epsilon^{2},
\end{equation}
where $\varrho_{R}(x)=\varrho(x/R)$ with $\varrho:\R^3 \rightarrow [0,1]$ a smooth cut-off function satisfying 
\begin{align} \label{defi-varrho}
\varrho(x) = \left\{
\begin{array}{ccl}
1 &\text{if}& |x| \leq 1/2, \\
0 &\text{if}& |x| \geq 1.
\end{array}
\right. 
\end{align}
Using the fact (see Lemma \ref{Imporide}) that
\[
\partial_{t}(|u|^{2}+3\gamma|v|^{2})=-2\nabla\cdot\IM (\overline{u}\nabla u)
-6\nabla\cdot\IM (\overline{v}\nabla v),
\]
\eqref{Lrs}, and $\|\nabla \varrho_{R}\|_{L^\infty(\R^3)}\lesssim R^{-1}$, an integration by parts and the H\"older inequality yield
\[
\left| \partial_{t}\int\varrho_{R}(x) (|u(t,x)|^{2}+3\gamma|v(t,x)|^{2}) dx  \right|\lesssim R^{-1}.
\]
Taking $R$ sufficient large such that $R^{-1}\epsilon^{-\frac{1}{4}}\ll \epsilon^{2}$, we infer from \eqref{Tapro} that
\[
\left\|\int \varrho_{R}(x)(|u(\cdot,x)|^{2}+3\gamma|v(\cdot,x)|^{2})dx\right\|_{L_{t}^{\infty}(I_{2})}
\lesssim \epsilon^2.
\]
This inequality implies that
\begin{equation}\label{Igf}
\| \varrho_R u  \|_{L_{t}^{\infty}L^{2}_{x}(I_{2}\times \R^{3})}\lesssim \epsilon   \quad \mbox{and}\quad  \| \varrho_R v  \|_{L_{t}^{\infty}L^{2}_{x}(I_{2}\times \R^{3})}\lesssim \epsilon.
\end{equation}
Thanks to the radial Sobolev embedding \eqref{rad-sobo}, we have from \eqref{Lrs} and \eqref{Igf} that
\begin{align*}
\|u\|_{L^\infty_t L^3_x(I_2\times \R^3)} &\leq \|\varrho_R u\|_{L^\infty_tL^3_x(I_2\times \R^3)} + \|(1-\varrho_R) u\|_{L^\infty_tL^3_x(I_2\times \R^3)} \\
&\lesssim \|\varrho_R\|^{1/2}_{L^\infty_tL^2_x(I_2\times \R^3)} \|\varrho_R u\|^{1/2}_{L^\infty_tL^6_x(I_2\times \R^3)} \\
& + \|(1-\varrho_R)u\|^{1/3}_{L^\infty_t L^\infty_x(I_2\times \R^3)} \|(1-\varrho_R) u\|^{2/3}_{L^\infty_tL^2_x(I_2\times \R^3)} \\
&\lesssim \epsilon^{\frac{1}{2}} + R^{-\frac{1}{3}} \lesssim \epsilon^{\frac{1}{2}}
\end{align*}
provided that $R>\epsilon^{-\frac{3}{2}}$. A similar estimate holds for $v$. In particular, we get
\begin{equation}\label{Estils}
\|(u,v)\|_{L_{t}^{\infty}L^{3}_{x}\times L_{t}^{\infty}L^{3}_{x} (I_{2}\times \R^{3})} \lesssim\epsilon^{\frac{1}{2}}.
\end{equation}
Moreover, we have from the local theory that
\[
\|(u,v)\|_{L_{t}^2L^{\infty}_{x}\times L_{t}^2L^{\infty}_{x} (I_2\times \R^3)} + 
\|(u,v)\|_{L_{t}^{2}\dot{W}^{\frac{1}{2},6}_{x}\times L_{t}^{2}\dot{W}^{\frac{1}{2},6}_{x} (I_2\times \R^3)}
\lesssim (1+|I_2|)^{\frac{1}{2}} \lesssim \epsilon^{-\frac{1}{8}}.
\]
By Sobolev embedding and Strichartz estimates, we see that that
\begin{multline} \label{est-H2}
\|H_{2}\|_{L_{t}^{4}L^{6}_{x}\times L_{t}^{4}L^{6}_{x}([T,\infty)\times \R^{3})} \\
\lesssim
\|(u,v)\|_{L^\infty_t L^3_x \times L^\infty_t L^3_x (I_2\times \R^3)} \|(u,v)\|_{L^2_tL^\infty_x \times L^2_t L^\infty_x (I_2\times \R^3)} \|(u,v)\|_{L^2_t \dot{W}^{\frac{1}{2},6}_x \times L^2_t \dot{W}^{\frac{1}{2},6}_x(I_2\times \R^3)} \lesssim \epsilon^{\frac{1}{4}}.
\end{multline}

On the other hand, the same argument developed in the proof of \eqref{cla1} shows that 
\begin{align} \label{est-H1}
\|H_{1}\|_{L_{t}^{4}L^{6}_{x}\times L_{t}^{4}L^{6}_{x}([T,\infty)\times \R^{3})}
\lesssim \epsilon^{\frac{1}{32}}.
\end{align}
Collecting \eqref{Cddf}, \eqref{est-H2}, and \eqref{est-H1}, we prove \eqref{scat-crit-app}, and the proof is complete.
\end{proof}

\section{Proofs of the main Theorems}\label{sec:proofs-main}

By exploiting the tools obtained in the previous parts of the paper, we are now able to prove the scattering for non-radial and radial solutions to \eqref{SNLS} given in Theorem \ref{Th1}. See \cite{MX,WY,XX} for analogous results for NLS systems of quadratic type. 

\subsection{Proof of the scattering results}

\begin{proof}[{Proof of Theorem \ref{Th1} for non-radial solutions}] 
It suffices to check the scattering criterion given in Proposition \ref{prop-scat-crit-non-rad}. To this end, we are inspired to \cite{XZZ}. Fix $a \in \R$ and let $\epsilon>0$ be a sufficiently small constant. Let $T_0=T_0(\epsilon)>0$ sufficiently large to be chosen later. We will show that there exists $t_0 \in (a, a+T_0)$ such that 
\begin{align} \label{scat-crit-non-rad-app}
\|(u,v)\|_{L^5_{t,x}\times L^5_{t,x}([t_0-\epsilon^{-\frac{1}{4}}, t_0] \times \R^3)} \lesssim \epsilon^{\frac{3}{140}}.
\end{align}
By Proposition \ref{Imnn}, there exist $T_0=T_0(\epsilon), J=J(\epsilon), R_0=R_0(\epsilon, u_0,v_0,\phi,\psi)$, $\sigma=\sigma(\epsilon)$, and $\eta=\eta(\epsilon)$ such that if $|\gamma-3|<\eta$, then 
\begin{multline*}
\frac{1}{JT_{0}}\int^{a+T_{0}}_{a}\int^{R_{0}e^{J}}_{R_{0}}\frac{1}{R^{3}}\int_{\R^{3}}
W_{\gamma}(\chi_{R}(\cdot-z)u(t),\chi_{R}(\cdot-z)v(t)) \\
\times K(\chi_{R}(\cdot-z)u^{\xi}(t),\chi_{R}(\cdot-z)v^{\xi}(t))dz\frac{dR}{R}dt \lesssim \epsilon.
\end{multline*}
It follows that there exists $R\in [R_{0}, e^{J}R_{0}]$ such that
\[
\frac{1}{T_{0}}\int^{a+T_{0}}_{a}\frac{1}{R^{3}}\int_{\R^{3}}
W_{\gamma}( \chi_{R}(\cdot-z)u(t),\chi_{R}(\cdot-z)v(t))K(\chi_{R}(\cdot-z)u^{\xi}(t),\chi_{R}(\cdot-z)v^{\xi}(t))dzdt
\lesssim \epsilon.
\]
In particular,
\[
\frac{1}{T_{0}}\int^{a+T_{0}}_{a}\frac{1}{R^{3}}\int 
 \|\chi_{R}(\cdot-z)u(t)\|^{2}_{L^{2}} \|\nabla\(\chi_{R}(\cdot-z)u^{\xi}(t)\)\|^{2}_{L^{2}}dzdt
\lesssim \epsilon
\]
and similarly for $v$. By the change of variable $z=\frac{R}{4}(w+\theta)$ with $w\in \Z^3$ and $\theta \in [0,1]^3$, we deduce from the integral mean value theorem and Fubini's theorem that there exists $\theta\in [0,1]^{3}$ such that
\[
\frac{1}{T_{0}}\int^{a+T_{0}}_{a}\sum_{w\in \Z^{3}}
\left\|\chi_{R} \(\cdot-\frac{R}{4}(w+\theta)\)u(t)\right\|^{2}_{L^{2}}
 \left\|\nabla\(\chi_{R}\(\cdot-\frac{R}{4}(w+\theta)\)u^{\xi}(t)\)\right\|^{2}_{L^{2}}dt
\lesssim \epsilon.
\]
By spliting the interval $[a+T_0/2, a+3T_0/4]$ into $T_0 \epsilon^{\frac{1}{4}}$ subintervals of the same length $\epsilon^{-\frac{1}{4}}$, we infer that there exists $t_0 \in [a+T_0/2, a+3T_0/4]$ such that $I_0:=[t_0-\epsilon^{-\frac{1}{4}}, t_0]\subset (a,a+T_0)$ and
\begin{equation}\label{Csit}
\int_{I_0}\sum_{w\in \Z^{3}}
\left\|\chi_{R}\(\cdot-\frac{R}{4}(w+\theta)\)u(t)\right\|^{2}_{L^{2}}
 \left\|\nabla\(\chi_{R}\(\cdot-\frac{R}{4}(w+\theta)\)u^{\xi}(t)\)\right\|^{2}_{L^{2}}dt
\lesssim \epsilon^{\frac{3}{4}}.
\end{equation}
In particular, by the classical Gagliardo-Nirenberg inequality 
\[
\|f\|^{4}_{L^{3}}\lesssim \|f\|^{2}_{L^{2}}\|\nabla f\|^{2}_{L^{2}},
\] 
we obtain
\begin{equation}\label{dir}
\int_{I_0}\sum_{w\in \Z^{3}}
\left\|\chi_{R}\(\cdot-\frac{R}{4}(w+\theta)\)u(t)\right\|^{4}_{L^{3}}\lesssim\epsilon^{\frac{3}{4}}.
\end{equation}
On the other hand, by using the H\"older inequality and the Sobolev embedding, we get
\begin{align}
\sum_{w\in \Z^{3}}
&\left\|\chi_{R}\(\cdot-\frac{R}{4}(w+\theta)\)u(t)\right\|^{2}_{L^{3}} \nonumber\\
&\lesssim 
\sum_{w\in \Z^{3}}\left\|\chi_{R}\(\cdot-\frac{R}{4}(w+\theta)\)u(t)\right\|_{L^{2}}
\left\|\chi_{R}\(\cdot-\frac{R}{4}(w+\theta)\)u(t)\right\|_{L^{6}} \nonumber \\
&\leq \Big( \sum_{\omega \in \Z^3} \Big\| \chi_R\Big(\cdot -\frac{R}{4}(w+\theta)\Big) u(t)\Big\|_{L^2}^2 \Big)^{1/2} \Big( \sum_{w\in \Z^3} \Big\|\chi_R \Big(\cdot-\frac{R}{4}(w+\theta) \Big) u(t)\Big\|^2_{L^6} \Big)^{1/2} \nonumber \\
&\lesssim \|u(t)\|_{L^2} \|u(t)\|_{H^1} \lesssim 1.\label{Oinm}
\end{align}
For the last line above we used the following: by Sobolev,
	\begin{align*}
	\sum_{w \in \Z^3} \Big\|\chi_R &\Big(\cdot-\frac{R}{4}(w+\theta)\Big) u(t)\Big\|^2_{L^6} \\
	&\lesssim \sum_{w \in \Z^3} \Big\|\chi_R \Big(\cdot-\frac{R}{4}(w+\theta)\Big) \nabla u(t)\Big\|^2_{L^2}+ \frac{1}{R^2} \Big\|(\nabla \chi)_R \Big(\cdot-\frac{R}{4}(w+\theta)\Big) u(t)\Big\|^2_{L^2} \\
	&\lesssim \|\nabla u(t)\|^2_{L^2} + \frac{1}{R^2\sigma^2} \|u(t)\|^2_{L^2} \lesssim \|u(t)\|^2_{H^1}
	\end{align*}
	as  $|\nabla\chi| \lesssim \sigma^{-1}$ and $R>R_0 =\epsilon^{-1} =\sigma^{-1}$ (see the end of the proof of Proposition \ref{Imnn}).
It follows from \eqref{dir}, \eqref{Oinm}, and the almost orthogonality that 
\begin{align}\nonumber
\| u  &\|^{3}_{L^{3}_{t,x}(I_0\times \R^{3})}\lesssim \int_{I_0}\sum_{w\in \Z^{3}}
\left\|\chi_{R}\(\cdot-\frac{R}{4}(w+\theta)\)u(t)\right\|^{3}_{L^{3} }\\\nonumber
&\leq \int_{I_0}\(\sum_{w\in \Z^{3}}
\left\|\chi_{R}\(\cdot-\frac{R}{4}(w+\theta)\)u(t)\right\|^{4}_{L^{3} }\)^{\frac{1}{2}}
\(\sum_{w\in \Z^{3}}
\left\|\chi_{R}\(\cdot-\frac{R}{4}(w+\theta)\)u(t)\right\|^{4}_{L^{2} }\)^{\frac{1}{2}}\\\nonumber
&\leq \(\int_{I_{0}}\sum_{w\in \Z^{3}}
\left\|\chi_{R}\(\cdot-\frac{R}{4}(w+\theta)\)u(t)\right\|^{4}_{L^{3} }\)^{\frac{1}{2}}
\(\int_{I_{0}}\sum_{w\in \Z^{3}}
\left\|\chi_{R}\(\cdot-\frac{R}{4}(w+\theta)\)u(t)\right\|^{4}_{L^{2}}\)^{\frac{1}{2}}\\\label{Nmbs}
&\lesssim \epsilon^{\frac{1}{4}}.
\end{align}
On the other hand, by Strichartz estimates, Sobolev embedding and standard continuity argument, we deduce that
\[ 
\|u\|_{L^{10}_{t,x}(I_{0}\times \R^{3})} \lesssim \left\langle I_{0}\right\rangle^{\frac{1}{10}}. 
\]
This inequality, \eqref{Nmbs}, and interpolation imply that
\[ 
\|u\|_{L^{5}_{t,x}(I_{0}\times \R^{3})}\lesssim  
\|u\|^{\frac{3}{7}}_{L^{3}_{t,x}(I_{0}\times \R^{3})}  
\|u\|^{\frac{4}{7}}_{L^{10}_{t,x}(I_{0}\times \R^{3})} \lesssim \epsilon^{\frac{3}{140}}.
\]
Similarly, we have
\[ 
\|v\|_{L^{5}_{t,x}(I_{0}\times \R^{3})}\lesssim \epsilon^{\frac{3}{140}}.
\]
Therefore, \eqref{scat-crit-non-rad-app} holds, and the proof is complete.
\end{proof}

\begin{proof}[{Proof of Theorem \ref{Th1} for radial solutions}]
We fix $\epsilon>0$ and $R$ as in Proposition \ref{prop-scat-crit-rad}. From \eqref{Mirs} and the mean value theorem, we infer that there exist sequences of times $t_n\rightarrow\infty$
and radii $R_n\rightarrow\infty$ such that
\begin{align} \label{est-n}
\lim_{n \to \infty}\int_{|x|\leq {R_{n}}} \(|u(t,x)|^{\frac{10}{3}}+|v(t,x)|^{\frac{10}{3}}\)dx=0.
\end{align}
Choosing $n$ sufficiently large so that $R_{n}\geq R$, the H\"older inequality yields
\begin{align*}
\int_{|x|\leq R} \(|u(t,x)|^{2}+3\gamma|v(t,x)|^{2} \) dx\lesssim R^{\frac{3}{5}} \left[\(\int_{|x|\leq R_{n}}|u(t,x)|^{\frac{10}{3}}dx \)^{\frac{3}{5}}
+\(\int_{|x|\leq R_{n}}|v(t,x)|^{\frac{10}{3}}dx \)^{\frac{3}{5}}\right]
\end{align*}
which, by \eqref{est-n}, shows \eqref{Taos}. By Proposition \ref{prop-scat-crit-rad}, the solution scatters forward in time.
\end{proof}

\subsection{Proof of the blow-up results}

It remains to prove the blow-up results as stated  in Theorem \ref{theo-blow}. Let us start with the following observation.

\begin{lemma} \label{lem-nega-G}
	Let $\mu, \gamma>0$, and $(\phi,\psi) \in \Gc(0, 3\gamma, \gamma)$. Let $(u_0,v_0) \in H^1 \times H^1$ satisfy either $E_\mu(u_0,v_0)<0$ or if $E_\mu(u_0,v_0) \geq 0$, we assume that \eqref{cond-ener} and \eqref{cond-blow} hold. Let $(u,v)$ be the corresponding solution to \eqref{SNLS} with initial data $(u_0,v_0)$ defined on the maximal time interval $(-T_-, T_+)$. Then for $\vareps>0$ sufficiently small, there exists $c=c(\vareps)>0$ such that 
	\begin{align} \label{est-G}
	G(u(t),v(t)) + \vareps K(u(t),v(t)) \leq -c
	\end{align}
	for all $t\in (-T_-, T_+)$.
\end{lemma}

\begin{proof}
	If $E_\mu(u_0,v_0)<0$, then the conservation of energy implies that
	\begin{align*}
	G(u(t),v(t)) + \frac{1}{2} K(u(t),v(t)) &= 3 E_\mu(u(t),v(t)) - \frac{3}{2} M_\mu(u(t),v(t))  \\
	&\leq 3 E_\mu(u(t),v(t)) = 3 E_\mu(u_0,v_0).
	\end{align*}
	This shows \eqref{est-G} with $\vareps =\frac{1}{2}$ and $c=-3E_\mu(u_0, v_0)>0$. 
	
	We next consider the case $E_\mu(u_0,v_0)\geq 0$. In this case, we assume \eqref{cond-ener} and \eqref{cond-blow}. 
	By the same argument as in the proof of \cite[Theorem 4.6]{OP} using \eqref{cond-ener} and \eqref{cond-blow}, we have
	\[
	K(u(t),v(t)) M_{3\gamma}(u(t),v(t)) > K(\phi,\psi) M_{3\gamma}(\phi,\psi), \quad \forall t\in (-T_-,T_+).
	\]
	Moreover, by taking $\rho=\rho(u_0,v_0,\phi,\psi)>0$ such that
	\begin{align} \label{defi-rho}
	E_\mu(u_0,v_0) M_{3\gamma}(u_0,v_0) \leq \frac{1}{2}(1-\rho) E_{3\gamma}(\phi, \psi) M_{3\gamma}(\phi,\psi),
	\end{align}
	we can prove (see again the proof of \cite[Theorem 4.6]{OP}) the existence of $\delta = \delta(u_0,v_0,\phi,\psi)>0$ such that
	\begin{align} \label{est-solu-blow}
	K(u(t),v(t)) M_{3\gamma}(u(t),v(t)) \geq (1+\delta) K(\phi,\psi) M_{3\gamma}(\phi,\psi), \quad \forall t\in (-T_-,T_+).
	\end{align}
	Now for $\vareps>0$ small to be chosen later, we have from \eqref{defi-rho}, \eqref{est-solu-blow}, and \eqref{poho-iden} that
	\begin{align*}
	\Big( G(u(t),v(t)) &+ \vareps K(u(t),v(t)) \Big) M_{3\gamma}(u(t),v(t)) \\
	&= \Big( 3 E_\mu(u(t),v(t)) - \frac{3}{2} M_\mu(u(t),v(t)) -\Big(\frac{1}{2}-\vareps\Big) K(u(t),v(t)) \Big) M_{3\gamma}(u(t),v(t)) \\
	&\leq 3 E_\mu(u(t),v(t)) M_{3\gamma}(u(t),v(t)) - \Big(\frac{1}{2}-\vareps\Big) K(u(t),v(t)) M_{3\gamma}(u(t),v(t)) \\
	&= \frac{3}{2}(1-\rho) E_{3\gamma}(\phi,\psi) M_{3\gamma}(\phi,\psi) - \Big(\frac{1}{2}-\vareps\Big)(1+\delta) K(\phi,\psi) M_{3\gamma}(\phi,\psi) \\
	&=-\Big( \frac{1}{2}(\rho+\delta) - \vareps(1+\delta)\Big) K(\phi,\psi) M_{3\gamma}(\phi,\psi)
	\end{align*}
	for all $t\in (-T_-,T_+)$. By choosing $0<\vareps<\frac{\rho+\delta}{2(1+\delta)}$, the conservation of mass yields
	\[
	G(u(t),v(t)) +\vareps K(u(t),v(t)) \leq -\Big( \frac{1}{2}(\rho+\delta) - \vareps(1+\delta)\Big) K(\phi,\psi) \frac{M_{3\gamma}(\phi,\psi)}{M_{3\gamma}(u_0,v_0)}
	\]
	for all $t \in (-T_-,T_+)$. The proof is complete.
\end{proof}

We are now able to provide a proof of Theorem \ref{theo-blow}. To the best of our knowledge, the strategy of using an ODE argument -- when classical virial estimates based on the second derivative in time of (localized) variance  break down -- goes back to the work \cite{BHL}, where fractional radial NLS is investigated. See instead \cite{DF, IKN-NA} for some blow-up results for quadratic NLS systems.\\
	
\noindent {\it Proof of Theorem \ref{theo-blow}.}
	We only consider the case of radial data, the one for $\Sigma_3$-data is treated in a similar manner using \eqref{viri-est-cyli}. Let $(u_0,v_0) \in H^1\times H^1$ be radially symmetric and satisfy either $E_\mu(u_0,v_0)<0$ or if $E_\mu(u_0,v_0) \geq 0$, we assume that \eqref{cond-ener} and \eqref{cond-blow} hold. Let $(u,v)$ be the corresponding solution to \eqref{SNLS} defined on the maximal time interval $(-T_-,T_+)$. We only show that $T_+<\infty$ since the one for $T_-<\infty$ is similar. Assume by contradiction that $T_+=\infty$. By Lemma \ref{lem-nega-G}, we have for $\vareps>0$ sufficiently small, there exists $c=c(\vareps)>0$ such that
	\begin{align} \label{nega-G-app}
	G(u(t),u(t)) + \vareps K(u(t),v(t)) \leq -c
	\end{align}
	for all $t\in [0,\infty)$. On the other hand, by Lemma \ref{viri-est-rad}, we have for all $t\in [0,\infty)$,
	\begin{align} \label{viri-est-rad-app-1}
	\frac{d}{dt} M_{\varphi_R}(t)\leq 8G(u(t),v(t)) + CR^{-2} K(u(t),v(t)) + CR^{-2},
	\end{align}
	where $\varphi_R$ is as in \eqref{defi-varphi-R} and $M_{\varphi_R}(t)$ is as in \eqref{defi-M-varphi}. It follows from \eqref{nega-G-app} and \eqref{viri-est-rad-app-1} that for all $t\in [0,\infty)$,
	\begin{align*}
	\frac{d}{dt} M_{\varphi_R}(t) \leq -8c - 8\vareps K(u(t),v(t)) +CR^{-2} K(u(t),v(t)) + CR^{-2}.
	\end{align*}
	By choosing $R>1$ sufficiently large, we get
	\begin{align} \label{viri-est-rad-app-2}
	\frac{d}{dt}M_{\varphi_R}(t) \leq -4c -4\vareps K(u(t),v(t))
	\end{align}
	for all $t\in [0,\infty)$. Integrating the above inequality, we see that $M_{\varphi_R}(t) <0$ for all $t\geq t_0$  with some $t_0>0$ sufficiently large. We infer from \eqref{viri-est-rad-app-2} that
	\begin{align} \label{viri-est-rad-app-3}
	M_{\varphi_R}(t) \leq -4\vareps \int_{t_0}^t K(u(s),v(s)) ds
	\end{align}
	for all $t\geq t_0$. On the other hand, by the H\"older's inequality and the conservation of mass, we have
	\begin{align}
	|M_{\varphi_R}(t)| &\leq C \|\nabla \varphi_R\|_{L^\infty} \left( \|\nabla u(t)\|_{L^2} \|u(t)\|_{L^2} + \|\nabla v(t)\|_{L^2} \|v(t)\|_{L^2} \right) \nonumber \\
	&\leq C(\varphi_R, M_{3\gamma}(u_0,v_0)) \sqrt{K(u(t),v(t))}. \label{viri-est-rad-app-4}
	\end{align}
	From \eqref{viri-est-rad-app-3} and \eqref{viri-est-rad-app-4}, we get
	\begin{align} \label{viri-est-rad-app-5}
	M_{\varphi_R}(t) \leq -A \int_{t_0}^t |M_{\varphi_R}(s)|^2 ds
	\end{align}
	for all $t\geq t_0$, where $A=A(\vareps, \varphi_R, M_{3\gamma}(u_0,v_0))>0$. Set 
	\begin{align} \label{viri-est-rad-app-6}
	z(t):= \int_{t_0}^t |M_{\varphi_R}(s)|^2 ds, \quad t\geq t_0.
	\end{align}
	We see that $z(t)$ is  non-decreasing and non-negative. Moreover, 
	\[
	z'(t) = |M_{\varphi_R}(t)|^2 \geq A^2 z^2(t), \quad \forall t\geq t_0.
	\]
	For $t_1>t_0$, we integrate over $[t_1,t]$ to obtain
	\[
	z(t) \geq \frac{z(t_1)}{1-A^2z(t_1)(t-t_1)}, \quad \forall t\geq t_1.
	\]
	This shows that $z(t) \rightarrow +\infty$ as $t \nearrow t^*$, where
	\[
	t^*:= t_1 + \frac{1}{A^2 z(t_1)} >t_1.
	\]
	In particular, we have
	\[
	M_{\varphi_R}(t) \leq -Az(t) \rightarrow -\infty
	\]
	as $t\nearrow t^*$, hence $K(u(t), v(t)) \rightarrow +\infty$ as $t\nearrow t^*$. Thus the solution cannot exist for all time $t\geq 0$. The proof is complete.
	\hfill $\Box$

\appendix	
\section{Proofs of Lemmas \ref{lem-smal-scat}, \ref{lem-refi-GN-ineq}, \ref{L22}, and \ref{lem-coer-2}}\label{sec:app:A}
Let $I \subset \R$ be an interval containing zero. We recall that a pair of functions $(u,v)\in C(I, H^1(\R^3)) \times C(I,H^1(\R^3))$ is called a solution to the problem \eqref{SNLS} if $(u,v)$ satisfies the Duhamel formula
\[  
(u(t),v(t))=(\Sc_{1}(t)u_{0},\Sc_{2}(t)v_{0})+i\int^{t}_{0}(\Sc_{1}(t-s)F_{1}(s),\Sc_{2}(t-s)F_{2}(s))ds
\]
for all $t\in I$, where 
\begin{align} \label{F1F2}
\begin{aligned}
	F_{1}(s)&:=\(\frac{1}{9}|u(s)|^{2}+2|v(s)|^{2}\)u(s)+\frac{1}{3}\overline{u}^{2}(s)v(s),\\
	F_{2}(s)&:=\(9|v(s)|^{2}+2|u(s)|^{2}\)v(s)+\frac{1}{9}u^{3}(s).
\end{aligned}
\end{align}
The linear operators $\Sc_1$ and $\Sc_2$ introduced in \eqref{def:propagators} satisfy the following dispersive estimates: for $j=1,2$, and $2\leq r \leq \infty$,
\begin{equation}\label{Dpe}
\|\Sc_{j}(t) f\|_{L^r(\R^{3})}\lesssim |t|^{-\(\frac{3}{2}-\frac{3}{r}\)} \|f\|_{L^{r'}(\R^3)}, \quad f \in L^{r'}(\R^3)
\end{equation}
for all $t\ne 0$, which in turn yield the following Strichartz estimates: for any interval $I\subset \R$ and any Strichartz $L^2$-admissible pairs $(q,r)$ and $(m, n),$ i.e., pairs of real numbers satisfying
\begin{align} \label{Sch-adm}
\frac{2}{q}+\frac{3}{r}=\frac{3}{2}, \quad 2\leq r \leq 6.
\end{align} 
we have, for $j=1,2$,
\begin{align*}
	\| \Sc_{j}(t)f  \|_{L_{t}^{q}L^{r}_{x}(I\times\R^{3})}&\lesssim \|f\|_{L^2(\R^{3})}, \quad f \in L^2(\R^3),\\
	\left\| \int^{t}_{0} \Sc_{j}(t-s) F(s) ds \right\|_{L_{t}^{q}L^{r}_{x}(I\times\R^{3})}
	&\lesssim \| F\|_{L_{t}^{m'}L^{n'}_{x}(I\times\R^{3})}, \quad F \in L^{m'}_t L^{n'}_x(I \times \R^3),
\end{align*}
where $(m,m')$ and $(n,n')$ are H\"older conjugate pairs. We refer the readers to the boos \cite{Cazenave, LP, Tao} for a general treatment of the Strichartz estimates for NLS equations.\\

We are ready to prove Lemma \ref{lem-smal-scat}.
\begin{proof}[Proof of Lemma \ref{lem-smal-scat}]
From the Duhamel formula, we have
\[  
(u(t),v(t))=(\Sc_{1}(t-T)u(T),\Sc_{2}(t-T)v(T))+i\int^{t}_{T}(\Sc_{1}(t-s)F_{1}(s),\Sc_{2}(t-s)F_{2}(s))ds.
\]
By using Sobolev embedding, Strichartz estimates, and interpolation, we get
\begin{align*}
\| (u,v)  \|_{L_{t}^{4}L^{6}_{x}\times L_{t}^{4}L^{6}_{x}([T,\infty)\times \R^{3})} &\leq \| (\Sc_{1}(t-T)u(T),\Sc_{2}(t-T)v(T))  \|_{L_{t}^{4}L^{6}_{x} \times L_{t}^{4}L^{6}_{x}([T,\infty)\times \R^{3})} \\
& + C \|(F_1, F_2)\|_{L^2_t W^{1,\frac{6}{5}}_x \times L^2_t W^{1,\frac{6}{5}}_x([T,\infty)\times \R^{3})} \\
& \leq \| (\Sc_{1}(t-T)u(T),\Sc_{2}(t-T)v(T)) \|_{L_{t}^{4}L^{6}_{x}\times L_{t}^{4}L^{6}_{x}([T,\infty)\times \R^{3})} \\
& + C\| (u,v)  \|^{2}_{L_{t}^{4}L^{6}_{x}\times L_{t}^{4}L^{6}_{x}([T,\infty)\times \R^{3})}\|(u,v)\|_{L_{t}^{\infty}L^3_{x} \times L^\infty_t L^3_x([T,\infty)\times \R^{3})} \\
&\leq  \| (\Sc_{1}(t-T)u(T),\Sc_{2}(t-T)v(T))  \|_{L_{t}^{4}L^{6}_{x}\times L_{t}^{4}L^{6}_{x}([T,\infty)\times \R^{3})}\\
&+E\| (u,v)  \|^{2}_{L_{t}^{4}L^{6}_{x}\times L_{t}^{4}L^{6}_{x}([T,\infty)\times \R^{3})},
\end{align*}

Choosing $\epsilon_{\sd}=\epsilon_{\sd}(E)>0$ small enough, the standard continuity argument implies that if \eqref{Small-sc} holds, then
\[
\| (u,v)  \|_{L_{t}^{4}L^{6}_{x}\times L_{t}^{4}L^{6}_{x}([T,\infty)\times \R^{3})}\lesssim \epsilon_{\sd}.
\]
Now, for $0<\tau<t$, we have
\begin{align*}
\|(\Sc_{1}(t)u(t),\Sc_{2}(t)v(t))&-(\Sc_{1}(\tau)u(t),\Sc_{2}(\tau)v(\tau)) \|_{H^{1} \times H^1 }\\
&= \left\| \int^{t}_{\tau}(\Sc_{1}(-s)F_{1}(s),\Sc_{2}(-s)F_{2}(s))ds   \right\|_{H^{1}\times H^1}\\
&\lesssim \| (u,v)  \|^{2}_{L_{t}^{4}L^{6}_{x}\times L_{t}^{4}L^{6}_{x}([\tau,t]\times \R^{3})}\|(u,v)\|_{L_{t}^{\infty}H^{1}_{x}\times L^\infty_t H^1_x([\tau,t]\times \R^{3})} \rightarrow 0
\end{align*}
as $\tau$, $t\to\infty$. Therefore, $\left\{(\Sc_{1}(t)u(t),\Sc_{2}(t)v(t)) \right\}_{t\to \infty}$ is a Cauchy sequence in 
$H^{1} \times H^{1} $. In particular, the solution $(u,v)$ scatters in the positive time.
\end{proof}
In the following, we provide the proofs for  Lemmas \ref{lem-refi-GN-ineq}, \ref{L22}, and \ref{lem-coer-2}.
\begin{proof}[Proof of Lemma \ref{lem-refi-GN-ineq}]
By the sharp Gagliardo-Nirenberg inequality \eqref{GN-ineq}, $K(|f|,|g|) \leq K(f,g)$, and \eqref{opti-cons}, we get
\[
P(|f|,|g|)\leq \frac{1}{3}\(\frac{K(f,g) M_{3\gamma}(f,g)}{K(\phi,\psi) M_{3\gamma}(\phi,\psi)}\)^{\frac{1}{2}} K(f,g).
\]
Thus
\begin{align*}
P(|f|,|g|)&\leq \frac{1}{3}\inf_{\xi_{1},\xi_{2}\in \R^{3}}\(
\(\frac{K(e^{ix\cdot\xi_{1}}f,e^{ix\cdot\xi_{2}}g) M_{3\gamma}(f,g) }{K(\phi,\psi) M_{3\gamma}(\phi,\psi)}\)^{\frac{1}{2}} 
K(e^{ix\cdot\xi_{1}}f,e^{ix\cdot\xi_{2}}g)\)\\
&\leq\frac{1}{3}\inf_{\xi_{1},\xi_{2}\in \R^{3}}
\(\frac{K(e^{ix\cdot\xi_{1}}f,e^{ix\cdot\xi_{2}}g) M_{3\gamma}(f,g)}{K(\phi,\psi) M_{3\gamma}(\phi,\psi)}\)^{\frac{1}{2}} 
\times \inf_{\xi_{1},\xi_{2}\in \R^{3}} K(e^{ix\cdot\xi_{1}}f,e^{ix\cdot\xi_{2}}g),
\end{align*}
which implies \eqref{refi-GN-ineq}.
\end{proof}

\begin{proof}[Proof of Lemma \ref{L22}]
By \eqref{GN-ineq} and $\mu>0$, we have
\begin{align*}
E_\mu(u(t),v(t)) M_{3\gamma}(u(t),v(t)) &\geq \frac{1}{2} K(u(t),v(t)) M_{3\gamma}(u(t),v(t)) - C_{\opt} \( K(u(t),v(t)) M_{3\gamma}(u(t),v(t))\)^{\frac{3}{2}} \\
&=: G\(K(u(t),v(t)) M_{3\gamma}(u(t),v(t))\)
\end{align*}
for all $t\in (-T_-,T_+)$, where $G(\lambda):=\frac{1}{2} \lambda - C_{\opt} \lambda^{\frac{3}{2}}$. Using \eqref{opti-cons}, we compute
\[
G\( K(\phi,\psi) M_{3\gamma}(\phi,\psi)\) = \frac{1}{6} K(\phi,\psi) M_{3\gamma}(\phi,\psi) = \frac{1}{2} E_{3\gamma}(\phi,\psi) M_{3\gamma}(\phi,\psi).
\]
By the conservation of mass and energy, and \eqref{cond-ener}, we have
\begin{align*}
G\( K(u(t),v(t)) M_{3\gamma}(u(t),v(t))\) &\leq E_\mu(u(t),v(t)) M_{3\gamma}(u(t),v(t)) \\
&= E_\mu(u_0,v_0) M_{3\gamma}(u_0,v_0) \\
&< \frac{1}{2} E_{3\gamma}(\phi,\psi) M_{3\gamma}(\phi,\psi) = G\( K(\phi,\psi) M_{3\gamma}(\phi,\psi)\)
\end{align*}
for all $t\in (-T_-,T_+)$. Using this and \eqref{cond-blow}, the continuity argument yields
\begin{align} \label{est-solu-gwp}
K(u(t),v(t)) M_{3\gamma}(u(t),v(t)) < K(\phi,\psi) M_{3\gamma}(\phi,\psi)
\end{align}
for all $t\in (-T_-,T_+)$. The blow-up alternative then implies that $T_-=T_+=\infty$. Next, by \eqref{GN-ineq}, \eqref{opti-cons}, and \eqref{est-solu-gwp}, we have
\[
P(u(t),v(t))\leq \frac{1}{3}
\(\frac{K(u(t),v(t)) M_{3\gamma}(u(t),v(t))}{K(\phi,\psi) M_{3\gamma}(\phi,\psi)}\)^{\frac{1}{2}} 
K(u(t),v(t))\leq\frac{1}{3}K(u(t),v(t))
\]
for all $t\in \R$. It follows that
\begin{equation}\label{est-E}
E_{\mu}(u(t),v(t))=\frac{1}{2}\(K(u(t),v(t))+M_{\mu}(u(t),v(t))\)-P(u(t),v(t))\geq \frac{1}{6}K(u(t),v(t))
\end{equation}
which, by the conservation of energy, implies \eqref{est-K}.

\noindent From \eqref{est-E} and \eqref{opti-cons}, we see that
\begin{align}\label{Ess}
\begin{aligned}
K(u(t),v(t))M_{3\gamma}(u(t),v(t))&\leq 6E_{\mu}(u(t),v(t))M_{3\gamma}(u(t),v(t)) \\
&=6\(\frac{E_{\mu}(u(t),v(t))M_{3\gamma}(u(t),v(t))}{E_{3\gamma}(\phi,\psi) M_{3\gamma}(\phi,\psi)}\) E_{3\gamma}(\phi,\psi) M_{3\gamma}(\phi,\psi)\\
&=
\(\frac{E_{3\gamma}(u(t),v(t))M_{3\gamma}(u(t),v(t))}{\frac{1}{2} E_{3\gamma}(\phi,\psi) M_{3\gamma}(\phi,\psi)}\) K(\phi,\psi) M_{3\gamma}(\phi,\psi)
\end{aligned}
\end{align}
for all $t\in \R$. On the other hand, by \eqref{cond-ener}, there exists $\delta=\delta(u_0,v_0,\phi,\psi)>0$ such that
\[ 
E_{\mu}(u_{0},v_{0}) M_{3\gamma}(u_{0},v_{0}) \leq (1-\delta) \frac{1}{2} E_{3\gamma}(\phi,\psi) M_{3\gamma}(\phi,\psi).
\]
Then from \eqref{Ess} and the conservation laws of mass and energy, we obtain
\[
K(u(t),v(t))M_{3\gamma}(u(t),v(t))\leq (1-\delta)K(\phi,\psi)M_{3\gamma}(\phi,\psi)
\]
for all $t\in \R$. The proof is complete.
\end{proof}

\begin{proof}[Proof of Lemma \ref{lem-coer-2}]
It follows from straightforward calculations that $\|\Gamma_R f\|^2_{L^2 } \leq \|f\|^2_{L^2 }$ and
\[
\int \Gamma^2_R(x) |\nabla f(x)|^{2}dx=\int |\nabla (\Gamma_R(x) f(x))|^{2}dx + \int \Gamma_R(x)\Delta\Gamma_R(x) |f(x)|^{2}dx, \quad f \in H^1.
\]
As $\|\Delta \Gamma_R\|_{L^\infty} \lesssim R^{-2}$, we infer from \eqref{coer-1} and the conservation of mass that there exists a sufficiently large $R=R(\delta, u_0,v_0, \phi,\psi)$ so that 
\[
K\(\Gamma_{R}(\cdot-z)u(t),\Gamma_{R}(\cdot-z)v(t)\) M_{3\gamma}\(\Gamma_{R}(\cdot-z)u(t),\Gamma_{R}(\cdot-z)v(t)\) \leq \(1-\frac{\delta}{2}\)K(\phi,\psi) M_{3\gamma}(\phi,\psi)
\]
for all $t\in \R$. The refined Gagliardo-Nirenberg inequality \eqref{refi-GN-ineq} implies that 
\[
P\(\Gamma_{R}(\cdot-z)|u(t)|,\Gamma_{R}(\cdot-z)|v(t)|\)\leq \frac{1}{3}\(1-\frac{\delta}{2}\)^{\frac{1}{2}} K\(\Gamma_{R}(\cdot-z)e^{ix\cdot\xi_{1}}u(t),\Gamma_{R}(\cdot-z)e^{ix\cdot\xi_{2}}v(t)\)
\]
which in turn implies \eqref{coer-prop} with $\nu:= 1-\(1-\frac{\delta}{2}\)^{\frac{1}{2}}>0$.
\end{proof}

\section{Virial Identities}\label{sec:app:B}
This Appendix is devoted to the proof of the virial identities in Section \ref{sec:VM}.
\begin{proof}[Proof of Lemma \ref{Imporide}]
	Notice that
	\begin{equation}\label{Idf1}
	\partial_{t}(|u|^{2}+\gamma\beta|v|^{2})=2\RE (\overline{u}\partial_{t} u +\gamma\beta\overline{v}\partial_{t}v).
	\end{equation}
	Moreover, multiplying the equation \eqref{SNLS} with $(\overline{u}, \beta\overline{v})$ and taking the imaginary part, we have
	\begin{equation}\label{Idf22}
	\begin{split}
	\RE (\overline{u}\partial_{t} u +\gamma\beta\overline{v}\partial_{t}v)=
	-\IM (\overline{u}\Delta u +\beta\overline{v}\Delta v)-\IM \(\frac{1}{3}\overline{u}^{3}v+\frac{\beta}{9}{u}^{3}\overline{v}\) \\
	=-\IM (\overline{u}\Delta u +\beta\overline{v}\Delta v)+
	\frac{1}{3}\(1-\frac{\beta}{3}\)\IM (u^{3}\overline{v}).
	\end{split}
	\end{equation}
	Combining \eqref{Idf1} and \eqref{Idf22}, we infer that
	\begin{align*}
	\partial_{t}(|u|^{2}+\gamma \beta|v|^{2}) &=
	-2\IM (\overline{u}\Delta u +\beta\overline{v}\Delta v)+
	\frac{2}{3}\(1-\frac{\beta}{3}\)\IM (u^{3}\overline{v})\\
	&=-2\nabla\cdot\IM (\overline{u}\nabla u)	-2\beta\nabla\cdot\IM (\overline{v}\nabla v)+\frac{2}{3}\(1-\frac{\beta}{3}\)\IM (u^{3}\overline{v}),
	\end{align*}
	which implies \eqref{Idg}. On the other hand, we rewrite \eqref{SNLS} as
	\[
	\left\{
	\begin{array}{ccl}
	i\partial_{t} u+\Delta u &=&H,\\
	i\gamma\partial_{t} v+\Delta v &=&G,
	\end{array}
	\right. 
	\]
	where $H=H_1+H_2+H_3$ and $G=G_1+G_2+G_3$ with
	\begin{align*}
	H_{1}&	=u, & H_{2}&=-\(\frac{1}{9}|u|^{2}+2|v|^{2}\)u, & H_{3}&=-\frac{1}{3}\overline{u}^{2}v,\\
	G_{1}&	=\mu v, & G_{2}&=-(9|v|^{2}+2|u|^{2})v, & G_{3}&=-\frac{1}{9}u^{3}.
	\end{align*}
	It follows from straightforward computations that 
	\begin{align} \nonumber
	\partial_{t}\IM (\overline{u} \partial_ku+\gamma\overline{v} \partial_k v)&= 
	\frac{1}{2}\partial_{k}\Delta (|u|^{2}+|v|^{2})-2\partial_{j}\RE(\partial_j\overline{u} \partial_ku+\partial_j\overline{v} \partial_kv)\\
	& + (2\RE (\overline{H}\partial_{k}u)-\partial_{k}\RE(\overline{H} u))+
	(2\RE (\overline{G}\partial_{k}v)-\partial_{k}\RE (\overline{G} v)). \label{iden-prof}
	\end{align}
	A simple calculation leads to
	\[
	(2\RE (\overline{H}_{1}\partial_{k}u)-\partial_{k}\RE (\overline{H}_{1} u))+
	(2\RE(\overline{G}_{1}\partial_{k}v)-\partial_{k}\RE (\overline{G}_{1} v))=0.
	\]
	Moreover, since
	\begin{align*}
	\partial_{k}(|u|^{2}|v|^{2})&=2\RE(\overline{u}\partial_{k} u)|v|^{2}+2\RE(\overline{v}\partial_{k} v)|u|^{2}\\
	\partial_{k}(|u|^{4})&=4|u|^{4}\RE(\overline{u}\partial_{k} u), \quad 
	\partial_{k}(|v|^{4})=4|v|^{4}\RE (\overline{v}\partial_{k} v),	
	\end{align*}
	we obtain that
	\[
	(2\RE(\overline{H}_{2}\partial_{k}u)-\partial_{k}\RE(\overline{H}_{2} u))+
	(2\RE (\overline{G}_{2}\partial_{k}v)-\partial_{k}\RE(\overline{G}_{2} v))=
	\frac{1}{18}|u|^{4}+\frac{9}{2}|v|^{4} +2|u|^{2}|v|^{2}.
	\]
	Finally, as
	\[
	\partial_{k}\RE(\overline{u}^{3} v)=3\RE (\overline{u}^{2}v \partial_{k}\overline{u})+\RE(\overline{u}^{3} \partial_{k}v),
	\]
	it follows that
	\[
	(2\RE(\overline{H}_{3}\partial_{k}u)-\partial_{k}\RE(\overline{H}_{3} u))+
	(2\RE (\overline{G}_{3}\partial_{k}v)-\partial_{k}\RE(\overline{G}_{3} v))=
	\frac{2}{9}\partial_{k}\RE(\overline{u}^{3} v).
	\]
	Collecting the above identities, we obtain
	\begin{align*}
	(2\RE (\overline{H}\partial_{k}u)-\partial_{k}\RE(\overline{H} u))+
	(2\RE (\overline{G}\partial_{k}v)-\partial_{k}\RE (\overline{G} v))=
	2\partial_{k}N(u,v),
	\end{align*}
	which, together with \eqref{iden-prof}, shows \eqref{Idn}. The proof is complete.
\end{proof}

\begin{proof}[Proof of Corollary \ref{rem-viri-iden}]
The proof of the identity \eqref{eq:variance} is straightforward. The relation \eqref{cor:ii} comes from the fact that
		\[
		\partial_j = \frac{x_j}{r} \partial_r, \quad \partial^2_{jk} = \left( \frac{\delta_{jk}}{r} - \frac{x_jx_k}{r^3} \right) \partial_r + \frac{x_j x_k}{r^2} \partial^2_r,
		\]
for radial function. Hence 
		\begin{align*}
		\rea \int \partial^2_{jk} \varphi(x) \partial_j \overline{u} (t,x) \partial_k u(t,x) dx 
		= \int  \frac{\varphi'(r)}{r} |\nabla u(t,x)|^2 dx + \int  \left(\frac{\varphi''(r)}{r^2}-\frac{\varphi'(r)}{r^3}\right) |x \cdot \nabla u(t,x)|^2 dx, 
		\end{align*}
		where $r=|x|$, which in turn implies  \eqref{cor:iii}.

If $\varphi$ is radial and $(u,v)$ as well, 
		\begin{align*}
		\frac{d}{dt} \Mcal_\varphi(t) &= -\int \Delta^2 \varphi(x) (|u|^2 + |v|^2)(t,x) dx  + 4 \int \varphi''(r) (|\nabla u|^2 + |\nabla v|^2)(t,x) dx \\
		& -4\int  \Delta \varphi(x) N(u,v)(t,x)dx. 
		\end{align*}

From the choice of the function $\varphi(x) = \psi(y) + z^2,$ we have
		\begin{align*}
		\frac{d}{dt} \Mcal_\varphi(t) &= -\int\Delta^2_y \psi(y) (|u|^2 +  |v|^2)(t,x) dx + 4 \rea \int\partial^2_{jk} \psi(y) (\partial_j \overline{u}\partial_k u + \partial_j \overline{v} \partial_k v)(t,x) dx \\
		& + 8\left(\|\partial_z u(t)\|^2_{L^2} + \|\partial_z v(t)\|^2_{L^2}\right) - 8 P(u(t),v(t))  -4\int  \Delta_y \psi(y) N(u,v)(t,x)dx	 		
		\end{align*}
which in turn gives \eqref{cor:iv}.
\end{proof}


\end{document}